\documentclass[a4paper;11pt]{amsart}
\usepackage{mathrsfs}
 \usepackage[arrow,matrix]{xy}
\usepackage{amsmath,amssymb,amscd,bbm,amsthm,mathrsfs}
\newtheorem{thm}{Theorem}[section]
\newtheorem{lem}{Lemma}[section]
\newtheorem{cor}{Corollary}[section]
\newtheorem{prop}{Proposition}[section]
\newtheorem{rem}{Remark}[section]

\theoremstyle{definition}

\begin{document}
\numberwithin{equation}{section}

 \title[On the weighted $L^2$  estimate for the  $k$-Cauchy-Fueter operator and the  $k$-Bergman   kernel ]{
On the weighted $L^2$  estimate for the  $k$-Cauchy-Fueter operator    and the weighted  $k$-Bergman   kernel}
\author {Wei Wang}
\begin{abstract} The $k$-Cauchy-Fueter operators, $k=0,1,\ldots$,
are   quaternionic counterparts of the Cauchy-Riemann operator   in the theory of several complex variables. The weighted $L^2$ method  to solve Cauchy-Riemann  equation is   applied to find the canonical solution  to the  non-homogeneous
$k$-Cauchy-Fueter equation  in a weighted $L^2$-space,  by establishing the weighted $L^2$ estimate. The weighted $k$-Bergman space is the space of weighted $L^2$ integrable functions annihilated by the $k$-Cauchy-Fueter operator, as the counterpart  of the Fock space  of weighted $L^2$-holomorphic  functions on $\mathbb{C}^n$. We introduce the $k$-Bergman orthogonal projection to this closed subspace, which can be  nicely expressed  in terms of the   canonical solution operator,  and its matrix kernel function. We also find the  asymptotic decay for this  matrix kernel function.
\end{abstract}

\thanks{Department of Mathematics, Zhejiang University, Zhejiang 310027, PR China,
Email: wwang@zju.edu.cn.}
\thanks{Supported by National Nature Science
Foundation
  in China (No. 11571305)}
 \maketitle

\section{Introduction}
 The $k$-Cauchy-Fueter operators over   $\mathbb{R}^{4n}$
 \begin{equation*}\label{eq:quaternionic-complex-op}\begin{split}
 {\mathscr D}_0^{(k)}: C^\infty \left(\mathbb{R}^{4n},  \odot^{k }\mathbb{C}^{2 }  \right) &
\longrightarrow C^\infty \left(\mathbb{R}^{4n}, \odot^{k-1
}\mathbb{C}^{2 } \otimes  \mathbb{C}^{2n}  \right) ,\end{split}
\end{equation*}$k=0,1,\ldots$,
 are  quaternionic counterparts of the Cauchy-Riemann operator $\overline{\partial}$  in the theory of several complex variables, where $\odot^{p}\mathbb{C}^{2 } $ is the $p$-th symmetric tensor product of $ \mathbb{C}^{2 } $. If we write a vector  in the quaternionic space $\mathbb H^{n}$ as
$\mathbf{q}=(\mathbf{q}_{0},\ldots,\mathbf{q}_{n-1})$,
  the  usual {\it Cauchy-Fueter
operator} is
defined as
\begin{equation*} \label{eq:CF}  \mathscr  D  :C^1( \mathbb H^{n}, \mathbb H )\rightarrow C (\mathbb H^{n}, \mathbb H^{n}),\qquad \qquad   \mathscr D f=\left(\begin{array}{c} \overline\partial_{
{\mathbf{q}}_0}f\\\vdots\\  \overline\partial_{ {\mathbf{q}}_{n-1  }}f\end{array}\right),
\end{equation*}
for $f\in C^1(\mathbb H^{n}, \mathbb H )$,  where $
 \overline{\partial}_{\mathbf{q}_{l}}
 =\partial_{x_{4l+1}}+\mathbf{i}\partial_{x_{4l+2}}+\mathbf{j}\partial_{x_{4l+3}}+\mathbf{k}\partial_{x_{4l+4}},
$ if we write $\mathbf{q}_{l}=x_{4l+1}+x_{4l+2}\mathbf{i}+x_{4l+3}\mathbf{j}+x_{4l+4}\mathbf{k}\in \mathbb H$,
$l=0,1,\ldots,n-1$. It is known that the  Cauchy-Fueter
operator coincides with the $1$-Cauchy-Fueter operator \cite{KW}. In the quaternionic case, we have a family of operators  acting on  $ \odot^{k} \mathbb{C}^2 $-valued functions, $k=0,1,\ldots$, because  SU$(2)$ as the group of unit quaternions  has  a family of irreducible representations $ \odot^{k} \mathbb{C}^2 $,  while $S^1$ as the group of unit complex numbers  has only one irreducible representation.
 The $k$-Cauchy-Fueter operators over  $\mathbb{R}^{4 }$ also have the origin in physics: they are
the elliptic version of {\it spin $  k/ 2$  massless field operators} over the
Minkowski space (cf. e.g. \cite{CMW} \cite{EPW} \cite{PR1} \cite{PR2}):
    $\mathscr D_0^{(1)}\phi=0$ corresponds to the Dirac-Weyl equation    whose solutions correspond to neutrinos;
  $\mathscr D_0^{(2)}\phi=0$ corresponds to  the  Maxwell equation   whose  solutions correspond to
photons;
  $\mathscr D_0^{(3)}\phi=0$ corresponds to   the Rarita-Schwinger
equation;
   $\mathscr D_0^{(4)}\phi=0$ corresponds to   linearized Einstein's equation
 whose  solutions correspond to  weak gravitational fields; etc..

 To  develop the function theory of several
quaternionic variables,  we need to solve  the  {\it non-homogeneous
$k$-Cauchy-Fueter equation}:
\begin{equation}\label{eq:w-kCF}
    {\mathscr D}_0^{(k)}u=f,
 \end{equation} where $u$ is $\odot^{k }\mathbb{C}^{2 } $-valued and $f$ is $\odot^{k-1
}\mathbb{C}^{2 } \otimes  \mathbb{C}^{2n}   $-valued. Under the identification
\begin{equation}\label{eq:iso}
   \odot^{k }\mathbb{C}^{2 }\simeq \mathbb{C}^{k+1},\qquad \odot^{k-1
}\mathbb{C}^{2 } \otimes  \mathbb{C}^{2n} \simeq \mathbb{C}^{2kn },
\end{equation}  ${\mathscr D}_0^{(k)}$ is a $2kn\times  (k+1)$-matrix valued differential operator of the first order with constant coefficients.
The equation (\ref{eq:w-kCF}) is overdetermined  and   its compatibility
condition is that $f$ is {\it
${\mathscr D}_1^{(k)}$-closed}, i.e.
 \begin{equation}\label{eq:compatibility}
    {\mathscr D}_1^{(k)}f=0,\end{equation}  where  ${\mathscr D}_1^{(k)}$ is    the second operator in  the $k$-Cauchy-Fueter complex:
\begin{equation}\label{eq:quaternionic-complex-diff}\begin{split}
0\rightarrow C^\infty \left(\mathbb{R}^{4n},  \mathscr V_0  \right) &
\xrightarrow{{\mathscr D}_0^{(k)}}C^\infty \left(\mathbb{R}^{4n}, \mathscr V_1 \right)
\xrightarrow{{\mathscr D}_1^{(k)}}  C^\infty \left(\mathbb{R}^{4n}, \mathscr V_2 \right) \xrightarrow{{\mathscr D}_2^{(k)}} \cdots  ,\end{split}
\end{equation}and
\begin{equation}\label{eq:quaternionic-complex-diff-V}\begin{split}
\mathscr V_0:=  \odot^{k }\mathbb{C}^{2 },  \qquad\mathscr V_1:=\odot^{k-1
}\mathbb{C}^{2 } \otimes  \mathbb{C}^{2n}, \qquad\mathscr V_2:= \odot^{k-2
}\mathbb{C}^{2 } \otimes \wedge^2 \mathbb{C}^{2n} .\end{split}
\end{equation} Here $\wedge^{2} \mathbb{C}^{2n} $ is the $2$-th exterior product of $ \mathbb{C}^{2n} $.
 These complexes
  play  the role  of Dolbeault complex   in  several complex variables,  and are now explicitly known \cite{Wa10} (cf. also \cite{Ba} \cite{bures}  \cite{bS} \cite{CSS}
  \cite{CSSS}).

  The author \cite{Wa08} \cite{Wa10} solved the  non-homogeneous
$k$-Cauchy-Fueter equation  in  $L^2$-space over   ${\mathbb R}^{4n}$ by using the method of classical harmonic analysis, and deduced Hartogs' phenomenon and integral representation formulae.
 In this paper, the weighted   $L^2$ method  to solve  the  $\overline{\partial}$  equation on $\mathbb{C}^n$ (see e.g. \cite{DA} \cite{H} \cite{MO} and references therein) is extended to solve  the  non-homogeneous
$k$-Cauchy-Fueter equation (\ref{eq:w-kCF}).   The  $L^2$ method is a general method to deal with overdetermined systems of linear differential equations when we can establish the necessary $L^2$ estimate, e.g. it is applied to
the  Dirac operator in Clifford analysis \cite{LCP}.
 The reason  to consider the weighted $L^2$-space  is as follows.
$f$ is called {\it $k$-regular} if ${\mathscr D}_0^{(k)}f=0$ in the sense of distributions.
It is known that the space of $k$-regular polynomials are   infinite dimensional (cf. \cite{KW}), and such functions are $L^2$-integrable with   Gaussian weight. This is similar to complex analysis, where one consider the space of $L^2$-integrable holomorphic functions with Gaussian  weight,     called {\it Fock space}. Without a weight,  a $L^2$-integrable holomorphic (or $k$-regular) function   must vanish.
Given a nonnegative function $\varphi$, called a {\it weighted function}, consider the Hilbert  space $L_\varphi^2( \mathbb{R}^{4n}, \mathbb{C}   )$   with the  weighted inner product
\begin{equation*}
   (u,v)_\varphi:=\int_{ \mathbb{R}^{4n}} u\overline{v}e^{-2\varphi}dV,
\end{equation*}
where $dV$ is the Lebegues  measure on $\mathbb{R}^{4n}$. For a complex linear space $\mathscr V$ with an inner product $\langle \cdot,\cdot\rangle$ (e.g. $\mathscr V=\odot^{k }\mathbb{C}^{2 } $ or $\odot^{k-1
}\mathbb{C}^{2 } \otimes  \mathbb{C}^{2n} $),  we   define $L_\varphi^2( \mathbb{R}^{4n}, \mathscr  V  )$ with the  weighted inner product
\begin{equation*}
   \langle f,g\rangle_\varphi:=\int_{ \mathbb{R}^{4n}} \langle f,g\rangle e^{-2\varphi}dV,
\end{equation*}and the  weighted norm $\|f\|_\varphi:=\langle f,f\rangle_\varphi^{\frac 12}$. The {\it weighted $k$-Bergman space} with respect to weight $\varphi=|x|^2$ is then defined as
 \begin{equation*}
    A^2_{(k)}(\mathbb{R}^{4n}, \varphi):=\left\{f\in L_\varphi^2( \mathbb{R}^{4n},  \odot^{k
}\mathbb{C}^{2 } );{\mathscr D}_0^{(k)}f=0\right\}.
 \end{equation*} It is infinite dimensional \cite{KW} because $k$-regular polynomials are integrable with respect to this weight.

In the sequel, we will drop the superscript for fixed $k$ for simplicity.  \begin{equation}\label{eq:quaternionic-complex-diff-L2}
   L_\varphi^2( {\mathbb R}^{4n},  \mathscr V_0 )\xrightarrow{{\mathscr D}_0}L_\varphi^2( {\mathbb R}^{4n},\mathscr V_1)\xrightarrow{{\mathscr D}_1}L_\varphi^2({\mathbb R}^{4n},\mathscr V_2)
\end{equation} is a complex, i.e. for any $u\in {\rm Dom} ({\mathscr D}_0 )$,
   \begin{equation*}
      {\mathscr D}_0u\in {\rm Dom}  ({\mathscr D}_1) \qquad {\rm and }\qquad {\mathscr D}_1{\mathscr D}_0u=0.
   \end{equation*}  Then if $f\in  L_\varphi^2( \mathbb{R}^{4n},  \mathscr V_1 )$ is ${\mathscr D}_1$-closed,
  the  nonhomogeneous  $k$-Cauchy-Fueter equation  (\ref{eq:w-kCF})
     has at most one solution $u\in {\rm Dom} ({\mathscr D}_0 ) $ orthogonal to $A^2_{(k)}(\mathbb{R}^{4n},\varphi)$. If it exists, it is called the {\it canonical solution} to the nonhomogeneous      $k$-Cauchy-Fueter equation (\ref{eq:w-kCF}).
 Consider the {\it associated Laplacian operator} $\Box_{\varphi }: L_\varphi^2( \mathbb{R}^{4n},  \mathscr V_1 )\longrightarrow L_\varphi^2( \mathbb{R}^{4n},  \mathscr V_1 )$ given by
\begin{equation*}
   \Box_{\varphi }:={\mathscr D}_0{\mathscr D}_0^*  + {\mathscr D}_1^*{\mathscr D}_1.
\end{equation*}

 \begin{thm} \label{thm:canonical} Suppose that $\varphi(x)=|x|^2$ and $k=2,3,\ldots$. Then
     \item[(1)]{ $\Box_\varphi$ has a bounded, self-adjoint and non-negative inverse $N_\varphi$ such that  }
       \begin{equation*}
          \|N_\varphi f\|_\varphi\leq \frac 14\|f\|_\varphi, \qquad {\rm for \hskip 3mm any }\quad f\in L_\varphi^2( \mathbb{R}^{4n},   \mathscr V_1 ).
       \end{equation*}

   \item[(2)]{    $ {\mathscr D}_0^*N_\varphi f $
        is  the canonical solution operator to the nonhomogeneous     $k$-Cauchy-Fueter equation (\ref{eq:w-kCF}), i.e. if $f\in {\rm Dom} ({\mathscr D}_1 ) $ is ${\mathscr D}_1$-closed, then
         \begin{equation*}
            {\mathscr
 D}_0 {\mathscr D}_0^*N_\varphi f=f \end{equation*}and $   {\mathscr D}_0^*N_\varphi f $ orthogonal to $A^2_{(k)}(\mathbb{R}^{4n},\varphi)$. Moreover,
        \begin{equation}\label{eq:canonical-est}
           \| {\mathscr D}_0^*N_\varphi f\|_\varphi\leq \frac 12 \|  f\|_\varphi,\qquad \| {\mathscr D}_1N_\varphi f\|_\varphi\leq  \frac 12\|  f\|_\varphi.
        \end{equation}
        }      \end{thm}

        The key step to prove this theorem is to establish the following   weighted $L^2$ estimate.
         \begin{thm}\label{thm:L2}
   Suppose that $\varphi(x)=|x|^2$ and $k=2,3,\ldots$. Then
    \begin{equation}\label{eq:L2-n}
 4 \left \| f \right\|^2_{  \varphi  }  \leq       \left\| {\mathscr D}_0^*f  \right\|^2_{\varphi } + \left\|{\mathscr D}_1f \right\|^2_{\varphi }
    \end{equation} for any $f\in {\rm Dom} ({\mathscr D}_0^*) \cap  {\rm Dom} ( {\mathscr D}_1)$.
   \end{thm}
The reason we only consider the weight $\varphi(x)=|x|^2$ is that the weighted $L^2$ estimate in this case is relatively easier.
   On $\mathbb{R}^{4n}$ for $n>1$, the operators ${\mathscr D}_0^{(0)}$ and ${\mathscr D}_1^{(1)}$ are differential operators of the second order,  and the weighted $L^2$ estimate
   is more difficult in these cases. While on $\mathbb{R}^{4 }$, the $k$-Cauchy-Fueter complexes for $k=0,1$ are trivial. So we   restrict to the case $k\geq 2$.

The   weighted $k$-Bergman space $A^2_{(k)}(\mathbb{R}^{4n}, \varphi)$ is a closed Hilbert  subspace.
  We call the orthogonal projection $P: L_\varphi^2( \mathbb{R}^{4n},  \odot^{k
}\mathbb{C}^{2 } ) \longrightarrow A^2_{(k)}(\mathbb{R}^{4n}, \varphi)$ the {\it  weighted $k$-Bergman projection}. It can be  nicely expressed  in terms of the the canonical solution operator as
 \begin{equation}\label{eq:Bergman-proj}
          Pf =f- {\mathscr D}_0^*N_\varphi{\mathscr D}_0f
       \end{equation}
       for $f\in{\rm Dom} ({\mathscr D}_0 ) $, as in the theory of several complex variables (cf. theorem 4.4.5 in \cite{CS}).

       If we use the first isomorphism in (\ref{eq:iso}),  a function in $L_\varphi^2( \mathbb{R}^{4n},  \odot^{k
}\mathbb{C}^{2 } )$ is $\mathbb{C}^{k+1 }$-valued.
       The  weighted $k$-Bergman projection $P$ has a kernel $K(x,y)$
 such that  the following integral formula   holds
 \begin{equation}\label{eq:Bergman-kernel}
    f(x)=\int_{\mathbb{R}^{4n}}K (x,y)f(y) e^{-2\varphi}dV
 \end{equation} for any $f\in A^2_{(k)}(\mathbb{R}^{4n}, \varphi)$.
 The kernel $K(x,y)$ is a $(k+1)\times(k+1)$-matrix valued function, which is $k$-regular in variables $x$ and anti-$k$-regular in variables $y$.

The main difference between the $k$-Cauchy-Fueter complexes and  Dolbeault complex   in the theory of several complex variables  is that there exist symmetric forms except for the exterior forms. The analysis of  exterior forms is classical, while the analysis of symmetric forms is relatively new. We can  handle components of a $\odot^{k }\mathbb{C}^{2 } $-  or $\odot^{k-1
}\mathbb{C}^{2 } \otimes  \mathbb{C}^{2n}  $-valued function. Such notations are used by physicists as two-spinor  notations for the   massless field  operators (cf. e.g. \cite{PR1} \cite{PR2} and references therein). They also appear in studying of quaternionic manifolds (cf. e.g. \cite{Wa-mfd} and references therein).

The weighted  $L^2$ estimate for the model case: $n=1$ and  $k=2$,   is obtain in section 2. The general case is proved in section 3.
Based on the weighted $L^2$ estimate,  Theorem \ref{thm:canonical} is proved in section 3.
In section 4, we establish a localized  a priori estimate for $\Box_\varphi$ and the    Caccioppoli-type estimate,  which hold  for many systems of PDEs of the divergence form. From these estimates and  the weighted $L^2$ estimate, we derive the asymptotic decay of the canonical solution $ {\mathscr D}_0^*N_\varphi  f$ to the nonhomogeneous      $k$-Cauchy-Fueter equation (\ref{eq:w-kCF}) when $f$   is compactly supported. Then by choosing suitable $f$ in (\ref{eq:Bergman-proj}), we find the asymptotic estimate for the weighted $k$-Bergman kernel from  the asymptotic behavior of the canonical solution.
\begin{thm}\label{thm:decay}Suppose that $\varphi(x)=|x|^2$ and $k=2,3,\ldots$. Then
   we have the following pointwise estimate for the weighted $k$-Bergman kernel: there exists $\varepsilon>0$
   only depending on $k,n $ such that
   \begin{equation}\label{eq:decay}
     |K(x,y)|\leq C e^{|x|^2+ |y|^2+\frac \varepsilon2(|x|+|y|) -\varepsilon|x-y|}
   \end{equation}
   for any $x,y\in \mathbb{R}^{4n}$ with $|x-y|>3$, and some constant $C>0$ only depending on $k,n ,\varepsilon$.
\end{thm}
 The first estimate for the  Bergman kernel of the weighted $L^2$-holomorphic functions over the complex plane $\mathbb{ C }$ is due to  Christ \cite{Ch}.
 The result of Christ was extended by Delin \cite{De} to several complex variables for  strict plurisubharmonic  weights. See also \cite{DA} \cite{MO} and references therein for  recent results. Our estimate is a little bit weaker than the complex case because we have  an extra factor $e^{\frac \varepsilon2(|x|+|y|)}$. But the estimate is the same when $|y|$ is larger compared to $|x|$ (cf. Remark \ref{rem:decay}).

I would like to thank the referee for valuable suggestions.
\section{The weighted $L^2$  estimate in the model case: $n=1$ and $k=2$}
\subsection{The complex vector fields  $Z_{AA'}$'s on $\mathbb{R}^{4n}$ and their formal adjoints}
To give the definition of the  $k$-Cauchy-Fueter operator, we need the following complex vector fields
  \begin{equation}\label{eq:k-CF} (Z_{AA'}):=\left(
                             \begin{array}{ll}
                               Z_{00'  } &  {Z}_{ 01' } \\
                               Z_{1 0' } &  {Z}_{11' } \\
                              \quad \vdots& \quad\vdots\\
                               Z_{ (2l )0' } & {Z}_{ (2l )1' } \\
                               Z_{ (2l+1)0'  } &  {Z}_{ (2l+1)1' } \\  \quad\vdots& \quad\vdots\\
                             \end{array}
                           \right):=\left(
                                      \begin{array}{ll}
                                     \partial_{x_1} +\textbf{i}\partial_{x_2}  & - \partial_{x_3} -\textbf{i}\partial_{x_4}  \\
                                       \partial_{x_3}-\textbf{i}\partial_{x_4}  &\hskip 3mm \partial_{x_1}-\textbf{i}\partial_{x_2}  \\
                                        \qquad \vdots&\qquad\vdots\\
                                       \partial_{x_{4l+1}} +\textbf{i}\partial_{x_{4l+2}}  & -\partial_{x_{4l+3}}-\textbf{i}\partial_{x_{4l+4} }  \\
                                       \partial_{x_{4l+3}} -\textbf{i}\partial_{x_{4l+4} }  &\hskip 3mm\partial_{x_{4l+1}} -\textbf{i}\partial_{x_{4l+2}}  \\ \qquad \vdots&\qquad\vdots\\
                                      \end{array}
                                    \right),
\end{equation} where $A=0,\ldots, 2n-1$, $A'=0',1'$. This is motivated by the embedding of  the quaternion  algebra   into the space of complex $2\times 2$-matrices:
\begin{equation*}
 x_{1}+\mathbf{i}x_{2}+\mathbf{j} x_{3}+\mathbf{k}x_{4}   \longmapsto
 \left(\begin{array}{rr} x_{1} +\textbf{i}x_{2}  & - x_{3} -\textbf{i}x_{4}  \\
                                       x_{3}-\textbf{i}x_{4}  &\hskip 3mm x_{1}-\textbf{i}x_{2}
 \end{array}\right).
 \end{equation*}
We will use
 \begin{equation}\label{eq:epsilon}  (\varepsilon_{A'B'})=\left( \begin{array}{cc} 0&
 1\\-1& 0\end{array}\right),\qquad  (\varepsilon^{A'B'}) =\left( \begin{array}{cc} 0&
- 1\\1& 0\end{array}\right)
 \end{equation}to raise or lower primed indices,
where $(\varepsilon^{A'B'})$ is the inverse of $(\varepsilon_{A'B'})$, i.e.,
$
  \sum_{B'=0' ,1'} \varepsilon_{A'B'}\varepsilon^{B'C'}=\delta_{A'}^{ C'}=\sum_{B'=0' ,1'}\varepsilon^{C'B'}\varepsilon_{B'A'}.
$
For example,
\begin{equation*}
   Z_{A}^{A'}=\sum_{B'=0' ,1'} Z_{AB'}\varepsilon^{B'A'}=Z_{A0'}\varepsilon^{0'A'}+Z_{A1'}\varepsilon^{1'A'}.
\end{equation*}
 In particular, we have
$
    Z_{A}^{0'}=Z_{A1'} ,   Z_{A}^{ 1'}=-Z_{A0'}
$ by
\begin{equation}\label{eq:varepsilon-anti}
   \varepsilon^{1'0'}=-\varepsilon^{0'1'}=1 ,\qquad  \varepsilon^{0'0'}= \varepsilon^{1'1'}=0
\end{equation}
 in (\ref{eq:epsilon}).
Then
\begin{equation}\label{eq:k-CF-raised} \left(Z_{A}^{A'}\right):=\left(
                             \begin{array}{ll}
                               Z_{0}^{0'  } &  {Z}_{ 0}^{1' } \\
                               Z_{1}^{ 0' } &  {Z}_{1}^{1' } \\
                               \quad \vdots& \quad\vdots\\
                               Z_{ (2n-2)}^{0' } & {Z}_{ (2n-2)}^{1' } \\
                               Z_{ (2n-1)}^{0'  } &  {Z}_{ (2n-1)}^{1' } \\
                             \end{array}
                           \right):= \left(
                             \begin{array}{ll}
                             {Z}_{ 01' } & - Z_{00'  } \\
                            {Z}_{11' } &  -Z_{1 0' } \\
                              \quad  \vdots& \quad\vdots\\
                             {Z}_{ (2n-2)1' } & -Z_{ (2n-2)0' } \\
                              {Z}_{ (2n-1)1' } & -Z_{ (2n-1)0'  } \\
                             \end{array}
                           \right).
\end{equation}

We also use
 \begin{equation} \label{eq:epsilon2} (\epsilon_{AB})= \left ( \begin{array}{rrrrr} 0&
 1&&&\\-1& 0&&&\\
 &&\ddots&&\\&&&0&
 1\\&&& -1& 0\end{array}\right)
 \end{equation} and $ (\epsilon^{A B }) $, the inverse of $(\epsilon_{AB})$,   to raise or lower unprimed indices, e.g.
$
    Z_{A'}^{A }=\sum_{B =0}^{2n-1}Z_{A'B }\epsilon^{B A } .
$
The advantage of using
  raising   indices is that the  adjoint  of $ Z_{A}^{A'}$  can  be written in a very simple form.
\begin{prop}\label{prop:unitary-gamma} (1) The formal adjoint operator $Z_\varphi^*$ of a complex vector field $Z$ is
\begin{equation*}
      Z_\varphi^*=-\overline{Z }+2\overline{Z  }\varphi.
\end{equation*}

 (2) We have    \begin{equation}\label{eq:Z-raise}
      Z^{AA'}=\overline{Z_{AA'}},
    \end{equation}
and   the formal adjoint operator of
 of $ Z_{A}^{A'}$ is
   \begin{equation}\label{eq:Z_AA'}
 \left ( Z_{A}^{A'}\right)_\varphi^*=  Z^{ A}_{A'} -2Z^{ A}_{A'}\varphi.
\end{equation}
\end{prop}
\begin{proof} (1) For a complex vector field $Z$, we have
\begin{equation*}
   (Z u,v)_\varphi=(u, Z_\varphi^*v)_\varphi.
\end{equation*}for $u,v\in C_0^\infty(\Omega,\mathbb{C})$.
This is because
\begin{equation*}
0= \int_{\Omega} Z(u \overline{v}  e^{-2\varphi})dV=\int_{\Omega} Z u\cdot \overline{v}  e^{-2\varphi} dV+\int_{\Omega} u\cdot Z \overline{v} \cdot e^{-2\varphi} dV-2\int_{\Omega} u \overline{v}\cdot Z\varphi \cdot e^{-2\varphi} dV.
\end{equation*}

(2)
 By raising indices,  $Z^{AA'}=\sum_{B =0}^{2n-1} \sum_{B'=0',1'} Z_{BB'}\epsilon^{BA}\varepsilon^{B'A'}$. It is direct from definition of $ Z_{AA'}$'s in (\ref{eq:k-CF}) to see that
 \begin{equation*}\begin{split}
   & \overline{Z_{00'}}=Z_{11'}=Z^{00'},\qquad  \overline{Z_{10'}}=-Z_{01'}=Z^{10'},\\&
   \overline{Z_{01'}}=-Z_{10'}=Z^{01'},\qquad  \overline{Z_{11'}}= Z_{00'}=Z^{11'}, \cdots,
    \end{split}
 \end{equation*}by  (\ref{eq:varepsilon-anti}) and similar relations for $\epsilon^{AB}$. Then $\overline{Z_{AA'}} =Z^{AA'}$.
Since  $( Z_{ A}^{A'})_\varphi^*=- \overline{Z_{ A}^{A'}} +2\overline{Z_{ A}^{A'}\varphi} $ by (1), and
\begin{equation}\label{eq:overline}
   \overline{Z_{ A}^{A'}}=\sum_{B'=0',1'}\overline{Z_{A B'}}\varepsilon^{B'A'}=-\sum_{B'=0',1'}Z^{ A B'} \varepsilon_{B'A'}=-Z^{ A}_{A'}
\end{equation}we get (\ref{eq:Z_AA'}). Here $\varepsilon^{B'A'}=- \varepsilon_{B'A'}$ by  (\ref{eq:epsilon}). \end{proof}

We will use the notations of the following complex differential operators:
\begin{equation}\label{eq:adjoint-delta00}
  \delta_{A'}^A :=  Z^{ A}_{A'} -2Z^{ A}_{A'}\varphi,
  \end{equation} for $A=0,\ldots, 2n-1$, $A'=0',1'$. Then we have $  ( Z_{ A}^{A'} )_\varphi^*=\delta_{A'}^A $ and
  \begin{equation}\label{eq:adjoint-delta}
    \left(Z_{ A}^{A'}u,v\right)_\varphi= \left( u,\delta_{A'}^A  v\right)_\varphi
  \end{equation}
for $u,v \in C_0^1(\Omega,\mathbb{C})$. By taking conjugate, we also have
  \begin{equation}\label{eq:delta-adjoint}
    \left(\delta_{A'}^A  u,v\right)_\varphi= \left( u, Z_{ A}^{A'} v\right)_\varphi.
  \end{equation}

\subsection{The weighted $L^2$  estimate in the model case  $n=1$ and  $k=2$ }

In this case,
\begin{equation}\label{eq:quaternionic-complex-diff-V-2}\begin{split}
\mathscr V_0:=  \odot^{2 }\mathbb{C}^{2 }\cong \mathbb{C}^3,  \qquad\mathscr V_1:= \mathbb{C}^{2 } \otimes  \mathbb{C}^{2 }\cong \mathbb{C}^4, \qquad\mathscr V_2:=  \wedge^2 \mathbb{C}^{2 } \cong \mathbb{C}^1 .\end{split}
\end{equation}
By definition,  $ \odot^{2
}\mathbb{C}^{2 } $ is a subspace of $ \otimes^{2
}\mathbb{C}^{2 } $, and
an element  $f$ of $  L_\varphi^2( \mathbb{R}^{4 },  \odot^{2
}\mathbb{C}^{2 } )$ has $4$ components
$f_{0'0'},f_{1'0'}, f_{0'1'}$ and $f_{1'1'}$ such that $ f_{1'0'}=f_{0'1'}$ . Its $L^2$ inner product is induced from that of $  L_\varphi^2( \mathbb{R}^{4 },  \otimes^{2
}\mathbb{C}^{2 } )$ by
\begin{equation*}
  \langle f,g\rangle_\varphi=\sum_{A',B'=0',1'} ( f_{A'B'}, {g_{A'B'}})_{ \varphi} = (f_{0'0'}, {g_{0'0'}})_{ \varphi}+2(f_{0'1'}, {g_{0'1'}})_{ \varphi}+(f_{1'1'},{g_{1'1'}})_{ \varphi};
\end{equation*}
 $f\in L_\varphi^2( \mathbb{R}^{4 },   \mathbb{C}^{2 } \otimes
 \mathbb{C}^{2 } )$
has $4 $ components $f_{A'A}$, $A=0,1, A'=0',1'$, and
\begin{equation*}
  \langle f,g\rangle_\varphi=\sum_{A=0,1} \sum_{A' =0',1'}  ( f_{A'A}, {g_{A'A}})_{ \varphi};
\end{equation*}
while
 $f\in L_\varphi^2( \mathbb{R}^{4 }, \wedge^2 \mathbb{C}^{2 }   )$
has  components  $f_{AB}$ with $f_{AB}=-f_{B A}$, among which there is only one nontrivial   (i.e. $ f_{00}=f_{11}=0$,  $ f_{01}=-f_{10} $),  and
 \begin{equation*}
  \langle f,g\rangle_\varphi=\sum_{A,B=0,1} ( f_{AB}, {g_{AB}})_{ \varphi}=2    ( f_{01}, {g_{01}})_{ \varphi}.
\end{equation*}

The operators in  the $2$-Cauchy-Fueter complex  over $\mathbb{R}^4$ are given by
\begin{equation}\label{eq:k-CF-0}\begin{split}
    ({\mathscr D}_0 \phi)_{A  'A }:& =\sum_{B'=0',1'}Z^{B'}_A\phi_{B' A'}=Z^{0'}_A\phi_{0'A' }+Z^{1'}_A\phi_{1'A' },\end{split}\end{equation}for
    $ \phi\in C^1( \mathbb{R}^{4 },  \odot^{2
}\mathbb{C}^{2 } )$
where  $ A =0,1, A'=0',1'$,    and \begin{equation}\label{eq:k-CF-1}\begin{split}
    ({\mathscr D}_1 \psi)_{AB }:&=2\sum_{A'=0',1'} Z^{A'}_{ [A}\psi_{B] A' } =\sum_{A'=0',1'}(Z^{A'}_{A}\psi_{B A' }-Z^{A'}_{B}\psi_{AA' })
\end{split}\end{equation}for
    $ \psi\in C^1( \mathbb{R}^{4 }, \mathbb{C}^{2 } \otimes
 \mathbb{C}^{2 } )$, where
 \begin{equation*} { {h}_{[A  B ] }}:=\frac 12 ({ {h}_{ A  B   }}-{ {h}_{B  A  }})
 \end{equation*}
is the antisymmetrisation. Here and in the sequel, we write $\psi_{ A  A ' }:=\psi_{A '  A}$   for convenience.   It is direct to see that
\begin{equation}\label{eq:exact}\begin{split}
    (\mathscr  D_1\mathscr D_0\phi)_{AB  } &=\sum_{A'=0',1'}\left(Z^{A'}_{ A}(\mathscr D_0\phi)_{  B A' }-Z^{A'}_{ B}(\mathscr D_0\phi)_{A   A' }\right)\\&=\sum_{A',C'=0',1'}\left( Z^{A'}_{
    A}Z^{C'}_B\phi_{C'A' }-Z^{A'}_{B}Z^{C'}_A\phi_{C'A'  }\right)=0
\end{split}\end{equation}
by relabeling  indices, $ \phi_{C'A' }=\phi_{A'C' }$ and   the commutativity $\nabla^{A'}_{ B}\nabla^{C'}_A=\nabla^{C'}_A\nabla^{A'}_{B}$,  as scalar differential operators of
constant complex coefficients  (cf. (2.11) in \cite{CMW}).

\begin{lem}\label{lem:sym} (1)  For any $h\in L_\varphi^2( \mathbb{R}^{4 },   \odot^2 \mathbb{C}^2)$ and $ H \in L_\varphi^2( \mathbb{R}^{4 },\otimes^2 \mathbb{C}^2)$, we have
   \begin{equation}\label{eq:sym}\sum_{A',B'} \left(  h_{A'B'} ,{H_{ A' B'  }}\right)_{ \varphi } =
    \sum_{A',B'} \left( h_{A'B'} , {H_{(A' B') }}\right)_{ \varphi } ,
   \end{equation}
where
 \begin{equation*}
    H_{(A' B')  }:=\frac 12 ({H_{ A' B'  }}+{H_{B' A' }})
 \end{equation*}
is the symmetrisation, i.e. $({H_{(A' B') }})\in L_\varphi^2( \mathbb{R}^{4 },\odot^2 \mathbb{C}^2)$.

\item[(2)]{For any $ h \in L_\varphi^2( \mathbb{R}^{4 },\wedge^2 \mathbb{C}^2)$ and $ H \in L_\varphi^2( \mathbb{R}^{4 },\otimes^2 \mathbb{C}^2)$, we have
    \begin{equation}\label{eq:antisym1}\sum_{A ,B} \left(  h_{ A B},{H_{ A B}}\right)_{ \varphi } =
    \sum_{A ,B   } \left(  h_{ A B  } ,{H_{ [A B]}}\right)_{ \varphi }.
   \end{equation}}

\item[(3)]{  For any $h,  H \in L_\varphi^2( \mathbb{R}^{4 },\otimes^2 \mathbb{C}^{2 })$, we have
   \begin{equation}\label{eq:antisym}
     \sum_{A ,B   }  \left(  h_{B A } ,{H_{ A  B  }}\right)_{ \varphi } =  \sum_{A ,B   }   \left(  h_{A B } ,{H_{ A  B   }}\right)_{ \varphi } -2  \sum_{A ,B   } \left(  h_{[A B] } ,{H_{ [ A  B]    }}\right)_{ \varphi } .
   \end{equation}}
\end{lem}
\begin{proof}
  (1) This is because
 \begin{equation*}
    \sum_{A',B'} h_{A'B'} \overline{H_{(A' B') }}=\frac 12 \sum_{A',B'} h_{A'B'} \left(\overline{H_{ A' B'  }}+\overline{H_{ B' A' }}\right)=  \sum_{A',B'} h_{A'B'} \overline{H_{ A' B'  }}
   \end{equation*} by changing indices and $h_{A'B'}=h_{B'A'}$.

   (2)  This is because\begin{equation} \label{eq:antisym-add} \sum_{ A,B } h_{ A B} \overline{H_{ A B}}=\frac 12 \sum_{A , B} h_{ A B}  \overline{(H_{ A B}- H_{B A})}  =
    \sum_{A ,B   } h_{A B} \overline{H_{[A B ] }}
   \end{equation}
 by changing indices and $h_{B A}=-h_{ A B}$.

 (3)  This is because
   \begin{equation}\label{eq:tensor-antisym}\begin{split}
      \sum_{A ,B   }  h_{B A } \overline{H_{ A  B  }}&= \sum_{A ,B   } h_{A B } \overline{H_{ A  B   }}+ \sum_{A ,B   }(h_{B A }- h_{A B }) \overline{H_{ A  B   }}
  \end{split} \end{equation}
  and the second term in the right hand side is
$
      - 2 \sum_{A ,B   } h_{[A B] }  \overline{H_{ A  B   }}=- 2 \sum_{A ,B   } h_{[A B] }  \overline{H_{ [A  B]   }}
  $
  by  the identity   (\ref{eq:antisym-add}).
\end{proof}

\begin{lem} \label{lem:T*-2} For $f\in C_0^\infty (\mathbb{R}^{4 },
\mathbb{C}^{2 } \otimes\mathbb{C}^{2 })$, we have
   \begin{equation}\label{eq:T*-2}
   ({\mathscr D}_0^*f)_{A' B'}=\sum_{A=0, 1} \delta_{(A ' }^{ A} f_{ B') A} .
   \end{equation}
\end{lem}
\begin{proof} For any $g\in C_0^\infty (\mathbb{R}^{4 },\odot^{ 2
}
 \mathbb{C}^{2 } ) $,  we have
   \begin{equation*}\begin{split}
   \langle {\mathscr D}_0 g,f\rangle_{\varphi }&= \sum_{A,A', B'}  \left(Z^{A ' }_{ A}  g_{ A' B'},   { f_{ B' A  }}\right)_{ \varphi }
    = \sum_{A,A', B'}  \left(   g_{ A' B'},  \delta_{ A ' }^{ A}  { f_{B' A   }}\right)_{ \varphi } \\&
     = \sum_{A',  B'}  \left(   g_{ A' B'}, \sum_{A} \delta_{( A ' }^{ A}  { f_{B')  A}}\right)_{ \varphi }=\langle g,{\mathscr D}_0^*f\rangle_{\varphi }
 \end{split}  \end{equation*}
 by using (\ref{eq:adjoint-delta}) and Lemma \ref{prop:unitary-gamma} (1).  Here we have to symmetrise $(A'B')$ in $\sum_{A} \delta_{  A ' }^{ A}  { f_{B'   A}} $ since  only after symmetrisation it becomes an element of $C_0^\infty (\mathbb{R}^{4 },\odot^{ 2
}
 \mathbb{C}^{2 } )$, i.e. a  $\odot^{2
}\mathbb{C}^{2 }$-valued function.
\end{proof}

      \begin{thm}
    Suppose that there exist a constant $c>0$ such that the weight $\varphi$ satisfies
    \begin{equation}\label{eq:pseudoconvex}
     \sum_{A,B,A',B'}    {Z^{ A ' }_{ B}} \overline{Z^{B ' }_{ A}}\varphi(x) \cdot \xi_{ A'A }   \overline{\xi_{   B' B }} \geq c \sum_{A ,A' }|\xi_{ A'A }|^2.
    \end{equation} for any $x\in  \mathbb{R}^{4n}$ and $(\xi_{ A' A  })\in \mathbb{C}^{2  } \otimes \mathbb{C}^{2  }$.
    Then we have the weighted $L^2$ estimate
    \begin{equation}\label{eq:L2}
 c \left \| f \right\|^2_{  \varphi  }  \leq       \left\| {\mathscr D}_0^*f  \right\|^2_{\varphi } + \left\|{\mathscr D}_1f \right\|^2_{\varphi },
    \end{equation} for  any $ f\in {\rm Dom} ({\mathscr D}_0^* )\cap  {\rm Dom}  ({\mathscr D}_1)$.
   \end{thm}
\begin{proof} By definition, we have
  $
  {\rm Dom} ({\mathscr D}_1 ): =\left\{ f\in L_\varphi^2( \mathbb{R}^{4n}, \mathscr V_1 );   {\mathscr D}_1 f\in L_\varphi^2( \mathbb{R}^{4n},\mathscr V_2 ) \right\}
$. Then ${\mathscr D}_1$ is  densely-defined since $C_0^\infty(\mathbb{R}^{4n},
 \mathscr V_1)$ is contained in its domain. It is also closed since differentiation is   continuous on distributions.  So is ${\mathscr D}_0^*$ as a differential operator given by (\ref{eq:T*-2}). Therefore it is sufficient to show   (\ref{eq:L2}) for $f\in C_0^\infty( \mathbb{R}^{4n},
 \mathscr V_1 )  $. It follows from the definition of ${\mathscr D}_0$ in (\ref{eq:k-CF-0}), ${\mathscr D}_0^*$ in Lemma \ref{lem:T*-2} and the definition of symmetrisation that
 \begin{equation}\label{eq:Sigma}\begin{split}2\langle {\mathscr D}_0^*f, {\mathscr D}_0^*f\rangle_{\varphi } &=2\left\langle{\mathscr D}_0 {\mathscr D}_0^*f,f\right\rangle_{\varphi }=    2 \sum_{B,B'}\left(\sum_{A' }Z_{B}^{A'}\sum_{ A }\delta_{(A  ' }^{ A}  f_{ B')A } ,   f_{   B' B  }\right)_{\varphi } \\& =  \sum_{A,B,A',B'}\left( Z_{B}^{A'}\delta_{ A  ' }^{ A}  f_{  B' A } ,   f_{ B' B  }\right)_{\varphi } +\sum_{A,B,A',B'}\left(   Z_{B}^{A'}\delta_{B ' }^{ A}  f_{ A'A },    f_{ B'B }\right)_{\varphi }:= \Sigma_0+\Sigma_1,\end{split}\end{equation}
  where \begin{equation}\label{eq:Sigma_0}\begin{split}\Sigma_0= \sum_{ A',B'}\left(\sum_{A  }\delta_{ A  ' }^{ A}  f_{ B' A } , \sum_{B }\delta^{B}_{A'}  f_{ B' B  }\right)_{\varphi }
  =\sum_{ A',B'} \left\|\sum_{A }  \delta_{ A  ' }^{ A}  f_{  B' A }\right\|_\varphi^2 \geq 0, \end{split}\end{equation}
and
  \begin{equation}\label{eq:Sigma_1} \begin{split}
  \Sigma_1 & =     \sum_{A,B,A',B'}\left\{\left( \delta_{B ' }^{ A} Z^{ A ' }_{ B} f_{ A A' }, f_{ B    B' }\right)_{\varphi }   +  \left(  [ Z^{ A ' }_{ B} ,\delta_{B ' }^{ A} ] f_{ A A' }, f_{ B    B' }\right)_{\varphi }\right\}\\&=   \sum_{A,B,A',B'}\left\{\left( Z^{ A ' }_{ B} f_{ A A' },Z^{B ' }_{ A} f_{ B    B' }\right)_{\varphi }+ 2\left( Z^{ A ' }_{ B}\overline{Z^{B ' }_{ A}}\varphi\cdot f_{ A A' }, f_{ B    B' }\right)_{\varphi }\right\} ,
   \end{split}\end{equation} by using the formal adjoint operator (\ref{eq:delta-adjoint}),  relabeling indices and using the commutator
    \begin{equation}\label{eq:commutator}
      \left [Z^{ A ' }_{ B},\delta_{B ' }^{ A}\right ]=- 2Z^{ A ' }_{ B}Z_{B ' }^{ A}\varphi=  2Z^{ A ' }_{ B}\overline{Z^{B ' }_{ A}}\varphi,
   \end{equation} which follows from  (\ref{eq:overline})-(\ref{eq:adjoint-delta00}) and the commutativity $Z^{ A ' }_{ B}Z_{B ' }^{ A}=Z_{B ' }^{ A}Z^{ A ' }_{ B}$ as scalar differential operators of
constant coefficients.
  The first summation in the right hand side of (\ref{eq:Sigma_1}) is equal to
 \begin{equation*}\begin{split}
    \sum_{A,B,A',B'}\left( Z^{ A ' }_{ B} f_{ A A' },Z^{B ' }_{ A} f_{ B    B' }\right)_{\varphi }
   & =  \sum_{A,B }\left( \sum_{ A' } Z^{ A ' }_{ B} f_{ A A' }, \sum_{ B'} Z^{B ' }_{ A} f_{ B    B' }\right)_{\varphi }\\&
  = \sum_{A,B }\left\|\sum_{ A '} Z^{ A ' }_{ A} f_{ B   A' }  \right\|^2_{\varphi }-2\sum_{A,B }\left\|\sum_{ A '} Z^{ A ' }_{ [A} f_{B]   A' } \right\|^2_{\varphi } \\&
  = \sum_{A,B }\left\|\sum_{ A '} Z^{ A ' }_{ A} f_{ B   A' }  \right\|^2_{\varphi }-\frac 12 \left\|\mathscr D_1 f   \right\|^2_{\varphi }
 \end{split} \end{equation*}
  by applying (\ref{eq:antisym}) with $h_{BA}=  \sum_{ A' } Z^{ A ' }_{ B} f_{ A A' }$, $H_{A B}=\sum_{ B'} Z^{B ' }_{ A} f_{ B    B' } $. Now substituting (\ref{eq:Sigma_0})-(\ref{eq:Sigma_1}) into (\ref{eq:Sigma}) and using the above identity, we get
  \begin{equation}\label{eq:bochner}\begin{split}
2\left\|\mathscr D_0^* f   \right\|^2_{\varphi }+  \frac 12  \left\|\mathscr D_1 f   \right\|^2_{\varphi }=&2 \sum_{A,B,A',B'}\left(  Z^{ A ' }_{ B}\overline{Z^{B ' }_{ A}} \varphi\cdot f_{ A A' }, f_{ B    B' }\right)_{\varphi }\\& +\sum_{ A',B'} \left\|\sum_{A }  \delta_{ A  ' }^{ A}  f_{  B' A }\right\|_\varphi^2+\sum_{A,B }\left\|\sum_{ A '} Z^{ A ' }_{ A} f_{ B   A' }  \right\|^2_{\varphi }.
  \end{split} \end{equation} Now the resulting estimate follows from the assumption (\ref{eq:pseudoconvex}) for $\varphi$.
\end{proof}
  \begin{rem}\label{rem:bochner} (1) We do not handle the term $\Sigma_0$ in (\ref{eq:Sigma_0}) by using    commutators. Because if we do so
\begin{equation*}\begin{split}
    \Sigma_0 &
        =\sum_{A , B,A',B'} ( \delta_{ A  ' }^{ A} Z^{ A  ' }_{B} f_{ A   B' }, f_{ B   B' })+( [Z^{ A  ' }_{B},  \delta_{ A  ' }^{ A}]  f_{ A   B' }, f_{ B   B' })\\&
    =\sum_{A,B,A',B'}\left( Z^{ A ' }_{ B} f_{ A   B' } , Z^{ A  ' }_{ A}f_{ B     B' }\right)_{\varphi }+2(  Z^{ A  ' }_{B}   \overline{Z^{ A  ' }_{ A}} \varphi  f_{ A   B' }, f_{ B   B' }),
       \end{split}\end{equation*}
       the first term in the right hand side above is quite difficult to control. But over $\mathbb{R}^4$ it can be controlled in terms of ${\mathscr D}_0^*f$ and  ${\mathscr D}_1f$. Based on such estimates, we can solve the  Neumann problem for the   $k$-Cauchy-Fueter complexes over  $k$-pseudoconvex domains in $\mathbb{R}^4$  (cf. \cite{Wa16}).

       (2) $\varphi=|x|^2$ satisfies the  assumption (\ref{eq:pseudoconvex}) for $\varphi$ with $c=4$ by the following Lemma \ref{lem:assumption}.
\end{rem}

  \section{The canonical solution operator to the nonhomogeneous     $k$-Cauchy-Fueter equation}

   \subsection{The weighted $L^2$  estimate in the general case}

Recall that the
{\it symmetric power} $ \odot^{k
}\mathbb{C}^{2 } $ is a subspace of $ \otimes^{k
}\mathbb{C}^{2 } $, and an element of $\odot^{k} \mathbb{C}^2 $ is given by a $2^k $-tuple  $
(  f_{A_1' \ldots A_k'})\in \otimes^{k} \mathbb{C}^2$ with $A_1' \ldots A_k'   =0',1' $, where
$  f_{A_1' \ldots A_k'} $ is invariant under  permutations of
subscripts, i.e.
\begin{equation*}
     f_{A_1' \ldots A_k'} =  f_{   A_{\sigma(1)}'\ldots A_{\sigma(k)} '},
 \end{equation*} for any $\sigma\in S_k$, the group of permutations of $k$ letters.  Note that $\dim (\odot^{k} \mathbb{C}^2) =k+1$ (cf. (\ref{eq:Dk})) while $\dim (\otimes^{k} \mathbb{C}^2) =2^k $.
An element of the exterior power $\wedge^{2}
\mathbb{C}^{2n} $ is given by  a tuple  $ (  f_{AB})$ with  $  f_{AB}= - f_{B A}$,
$A, B= 0,\ldots, 2n -1$. An element of $\odot^{k-1} \mathbb{C}^2\otimes \mathbb{C}^{2n} $ is given by a tuple  $
(  f_{A_2' \ldots A_k'A})\in \otimes^{k-1}\mathbb{C}^2\otimes \mathbb{C}^{2n}$, which
  is invariant under   permutations of $A_2' ,\ldots, A_k'$.
We will use  symmetrisation   of primed indices
\begin{equation}\label{eq:sym-0}\begin{split}
     f_{\cdots (A_1'\ldots A_k')\cdots}:&=\frac 1 {k!}\sum_{\sigma\in S_k}  f_{\cdots  A_{\sigma(1)}'\ldots A_{\sigma(k)} '\cdots}.
\end{split}\end{equation}

 The first two operators in  $k$-Cauchy-Fueter complex (\ref{eq:quaternionic-complex-diff})-(\ref{eq:quaternionic-complex-diff-V}) over $\mathbb{R}^{4n}$ are given by
\begin{equation}\label{eq:k-CF-}\begin{split}
    & ({\mathscr D}_0 f)_{ A_2'\ldots A_{k  }'A }:=\sum_{A_1'=0',1'}Z^{A_1'}_Af_{ A_1'A_2'\ldots A_{k }'  }=Z^{0'}_Af_{0'A_2'\ldots A_{k  }'  }+Z^{1'}_Af_{1'A_2'\ldots A_{k  }'  },\\&
    ({\mathscr D}_1 h)_{AB A_3'\ldots A_{k  }'}:=2 \sum_{A'=0',1'} Z^{A'}_{[ A}h_{B]A' A_3'\ldots A_{k  }' }=\sum_{A'=0',1'}\left(Z^{A'}_{ A}h_{BA' A_3'\ldots A_{k  }' }-Z^{A'}_{B}h_{ AA'A_3'\ldots A_{k  }' }\right ),
\end{split}\end{equation} for $f\in C^1( \mathbb{R}^{4n},  \mathscr  V_0   )$, $h\in  C^1( \mathbb{R}^{4n}, \mathscr  V_1   )$, where  $ A,B =0,1,\ldots , 2n-1$, $ A_2',\ldots, A_{k  }'=0',1'$. Here and in the sequel, we write $h_{ AA_2'A_3'\ldots A_{k  }' }:=h_{A_2'A_3'\ldots A_{k  }' A}$ for convenience. It is direct to check that
${\mathscr D}_1\circ {\mathscr D}_0=0$ as (\ref{eq:exact}).

The weighted inner product of $  L_\varphi^2( \mathbb{R}^{4n}, \mathscr  V_0  )$ is induced from that of $  L_\varphi^2( \mathbb{R}^{4n},  \otimes^k \mathbb{C}^2  )$.  Namely  we define
 \begin{equation*}
  \left \langle f, h\right\rangle_{\varphi }:=\sum_{   A_1', \ldots,A_k'}\left(  f_{     A_1'\ldots A_k' } ,   h_{  A_1'\ldots A_k' }\right)_{\varphi }
 \end{equation*}for $f, h\in L_\varphi^2( \mathbb{R}^{4n}, \mathscr  V_0  )$, and $\|f\|_\varphi= \left \langle f,  {f}\right\rangle_{\varphi }^{\frac 12}$.
We define the  weighted induced inner products of $ L_\varphi^2( \mathbb{R}^{4n}, \mathscr  V_1  )$ and $ L_\varphi^2( \mathbb{R}^{4n}, \mathscr  V_2 )$ similarly.

\begin{lem} \label{lem:T*} For $f\in C_0^\infty (\mathbb{R}^{4n},  \mathscr  V_1 )$, we have
   \begin{equation}\label{eq:T*}
   ({\mathscr D}_0^*f)_{A_1' A_2'\ldots A_k'}=\sum_{A=0}^{2n-1} \delta_{(A_1 ' }^{ A} f_{ A_2'\ldots A_k') A } .
   \end{equation}
\end{lem}
\begin{proof} For any $g\in C_0^\infty (\mathbb{R}^{4n}, \mathscr  V_0 )$ we have
   \begin{equation*}\begin{split}
   \langle {\mathscr D}_0 g,f\rangle_{\varphi }&= \sum_{A,A_2',  \ldots, A_k'}  \left(\sum_{ A_1' }  Z^{A_1 ' }_{ A}  g_{ A_1'  \ldots A_k'},   { f_{ A_2'\ldots A_k'A }}\right)_{ \varphi }
    = \sum_{A,  A_1',\ldots, A_k'}  \left(   g_{ A' A_1'\ldots A_k'},  \delta_{ A_1 ' }^{ A}  { f_{ A_2'\ldots A_k'A }}\right)_{ \varphi } \\&
     = \sum_{   A_1',\ldots, A_k'}  \left(   g_{   A_1'\ldots A_k'}, \sum_{A} \delta_{( A_1 ' }^{ A}  { f_{  A_2'\ldots A_k')  A }}\right)_{ \varphi }=\langle g,{\mathscr D}_0^*f\rangle_{\varphi }
 \end{split}  \end{equation*}
by using  (\ref{eq:adjoint-delta})  and symmetrisation
\begin{equation}\label{eq:bracket-omit}\sum_{ A_1', \ldots, A_k'
   } \left(   g_{  A_1' \ldots A_k' }, G_{  A_1' \ldots A_k' }\right)_{\varphi }=
  \sum_{A_1', \ldots, A_k'} \left(    g_{  A_1' \ldots A_k' },G_{ (A_1' \ldots A_k')}\right)_{\varphi }
\end{equation}
for any   $  g \in    L_\varphi^2(  \mathbb{R}^{4 n},  \odot^k \mathbb{C}^{2 }   )$,  $  G\in    L_\varphi^2(  \mathbb{R}^{4 n},  \otimes^k \mathbb{C}^{2 }   )$. Here we have to symmetrise indices $( A_1 ' \ldots A_k')$ in $\sum_{A} \delta_{ A_1 ' }^{ A}  { f_{  A_2'\ldots A_k'  A }}$  since  only after symmetrisation it becomes an element of  $ C_0^\infty (\mathbb{R}^{4n}, \mathscr  V_0 )$, i.e. a  $\odot^{k
}\mathbb{C}^{2 }$-valued function. (\ref{eq:bracket-omit}) is a generalization of Lemma \ref{lem:sym} (1). It holds because
\begin{equation*}
   R.H.S.=\frac 1 {k!}\sum_{A_1', \ldots, A_p'} \sum_{\sigma\in S_k}\left(    g_{  A_1' \ldots A_k' },G_{  A_{\sigma(1)}' \ldots A_{\sigma(1)}' }\right)_{\varphi }=\frac 1 {k!}\sum_{\sigma\in S_k}\sum_{A_1', \ldots, A_k'} \left(    g_{  A_{\sigma^{-1}(1)}' \ldots A_{\sigma^{-1}(k)}'  },G_{ A_1' \ldots A_k' }\right)_{\varphi }
\end{equation*}by relabeling  indices, which equals to L.H.S. by $g$ symmetric in the indices, i.e. $g_{  A_{\sigma^{-1}(1)}' \ldots A_{\sigma^{-1}(k)}'  }=g_{  A_{1}' \ldots A_{k}'  }$ for any permutation $\sigma$.
\end{proof}

{\it Proof of Theorem \ref{thm:L2}}. As in the model case $n=1$, $k=2$, it is sufficient to show the weighted $L^2$-estimate (\ref{eq:L2-n}) for $f\in C_0^\infty( \mathbb{R}^{4n},   \odot^{k-1}
 \mathbb{C}^{2 } \otimes\mathbb{C}^{2n})$. Recall that if $(F_{A_1'\ldots A_k'})\in \otimes^k \mathbb{C}^{2 }$ is symmetric in $A_2'\ldots A_k'$, then we have
 \begin{equation}\label{eq:sym-1}
   F_{(A_1'\ldots A_k')}=\frac 1k (F_{A_1'A_2'\ldots A_k'}+\cdots+ F_{A_s'A_2'\ldots A_1'\ldots A_k'}+\cdots+ F_{A_k'A_2'\ldots A_1' } ),
 \end{equation}
 by definition of symmetrisation (\ref{eq:sym-0}).
 Now we expand the symmetrisation to get
   \begin{equation*}\begin{split}
k\langle {\mathscr D}_0^*f, & {\mathscr D}_0^*f\rangle_{\varphi } =  k\langle {\mathscr D}_0 {\mathscr D}_0^*f,f\rangle_{\varphi } \\& = k \sum_{ B, A_2', \ldots,A_k'}\left (\sum_{   A_1' }Z_{B}^{A_1'}\sum_{ A }\delta_{(A_1 ' }^{ A}  f_{  A_2'\ldots A_k') A  } ,  f_{  A_2'\ldots A_k' B } \right)_{\varphi }\\& = \sum_{A,B, A_1', \ldots,A_k'} \left( Z_{B}^{A_1'} \delta_{ A_1 ' }^{ A} f_{ A_2'\ldots  A_k' A } ,    f_{ A_2'\ldots A_k'B }\right)_{\varphi }+  \sum_{A,B, A_1', \ldots,A_k'}\sum_{s=2}^k\left( Z_{B}^{A_1'}\delta_{ A_s ' }^{ A} f_{\ldots A_1'\ldots A_k' A } ,
   f_{ A_2'\ldots A_k' B  }\right)_{\varphi } \\&  =:\Sigma_0 +\Sigma_1, \end{split}\end{equation*} by the adjoint operator ${\mathscr D}_0^*$ in Lemma \ref{lem:T*}. Here we split the sum into the cases $s=1$ and $s\geq 2$ as in the model case (cf.  Remark \ref{rem:bochner}).  Note that
\begin{equation*}
   \Sigma_0 = \sum_{  A_1', \ldots,A_k'} \left(  \sum_{A}\delta_{ A_1 ' }^{ A} f_{  A_2'\ldots  A_k'  A } , \sum_{B} \delta_{ A_1 ' }^{ B}  f_{ A_2'\ldots A_k' B }\right)_{\varphi } \geq 0
\end{equation*} by using (\ref{eq:adjoint-delta}),  and
\begin{equation*}\begin{split}
    \Sigma_1  & =\sum_{s=2}^k \sum_{A,B, A_1', \ldots,A_k'}\left(\delta_{ A_s ' }^{ A} Z^{ A_1' }_{  B } f_{  A_2'\ldots A_1'\ldots A_k' A  } ,  f_{    A_2'\ldots A_k' B }\right)_{\varphi } +\left(\left [Z^{ A_1' }_{  B }, \delta_{ A_s ' }^{ A}\right ] f_{  A_2'\ldots A_1'\ldots A_k' A } , f_{ A_2'\ldots A_k' B }\right)_{\varphi } \\& =\Sigma_2' +\Sigma_2'' \end{split}\end{equation*}by using commutators.
For the second sum,
\begin{equation*}\begin{split}
   \Sigma_2''&= 2\sum_{s=2}^k\sum_{A,B, A_1', \ldots,A_k'} \left( Z^{ A_1' }_{  B }\overline{Z^{ A_s ' }_{ A}}\varphi \cdot f_{ A_2'\ldots A_1'\ldots A_k' A } , f_{ A_2'\ldots A_k'B }\right)_{\varphi }\\&= 8\sum_{s=2}^k\sum_{A,B, A_1', \ldots,A_k'} \left(\delta_{  B A} \delta_{A_1'A_s '} \cdot f_{ A_2'\ldots A_1'\ldots A_k'A } , f_{ A_2'\ldots A_k' B  }\right)_{\varphi }=8( {k-1})\|f\|_{\varphi }^2
\end{split}\end{equation*}for  $\varphi(x)=|x|^2$ by the following Lemma \ref{lem:assumption}  and $f$ symmetric in the primed indices. On the other hand,
\begin{equation*}\begin{split}
 \Sigma_2'&   = \sum_{s=2}^k\sum_{A,B, A_1', \ldots,\ldots, A_k'} \left( Z^{ A_1' }_{  B } f_{  A_2' \ldots A_1'\ldots A_k' A } , Z^{ A_s ' }_{ A} f_{ A_2'\ldots A_k' B }\right)_{\varphi }\\&= \sum_{s=2}^k\sum_{A,B  } \sum_{ \widehat{A}_1', \ldots,\widehat{A}_s',\ldots ,A_k'} \left(\sum_{  A_1' } Z^{ A_1' }_{  B } f_{    A_1' \ldots\widehat{A_s'}\ldots A_k' A  } , \sum_{  A_s' } Z^{ A_s ' }_{ A} f_{ A_s' A_2'\ldots \widehat{A_s'} \ldots A_k' B }\right)_{\varphi }\\&= (k-1) \sum_{B_3', \ldots,B_k'=0',1'}  \sum_{A,B}  \left( \sum_{A'} Z^{ A' }_{  B } f_{     A' B_3' \ldots  B_k' A } , \sum_{A'} Z^{A' }_{ A} f_{ A' B_3'\ldots B_k'  B  }\right)_{\varphi }\end{split}\end{equation*}  by $f$ symmetric in the primed indices and relabelling  indices. Then applying  Lemma \ref{lem:sym} (3) ((\ref{eq:tensor-antisym}) holds for $A,B=0,\ldots, 2n-1$) to the right hand side with $h_{BA}=\sum_{A'} Z^{ A' }_{  B } f_{ A    A' B_3' \ldots  B_k' } $ and $H_{AB}= \sum_{A'} Z^{ A'}_{ A} f_{ B  A' B_3'\ldots B_k' } $ for fixed $B_3', \ldots , B_k'$, we get
\begin{equation*}\begin{split}\Sigma_2'&
    =(k-1) \sum_{  B_3', \ldots,B_k'}  \sum_{A,B} \left\{ \left\|\sum_{A'} Z^{ A ' }_{  A} f_{B      A' B_3'\ldots B_k' }\right\|^2_{\varphi }-2  \left\|\sum_{A'} Z^{ A ' }_{ [ A} f_{B]  A'   B_3'\ldots B_k' }  \right\|_{\varphi }^2\right\}\\&
    =(k-1) \sum_{A,B, B_3', \ldots,B_k'} \left\|\sum_{A'}  Z^{ A ' }_{  A} f_{B    A'   B_3'\ldots B_p' }\right\|^2_{\varphi }-\frac {k-1}2  \left\| {\mathscr D}_1  f  \right\|^2_{\varphi }.
   \end{split}\end{equation*}
   Now we get
   \begin{equation*}
    k  \|{\mathscr D}_0^*f\|^2_{\varphi }+ \frac {k-1}2   \left\| {\mathscr D}_1  f  \right\|^2_{\varphi }\geq 8 (k-1)\|f\|_{\varphi }^2.
   \end{equation*}
   The estimate (\ref{eq:L2-n}) follows. \hskip 115 mm $\Box$

   \begin{lem}\label{lem:assumption} $ Z^{ A ' }_{ B} \overline{{Z^{B ' }_{ A}}} |x|^2  =4\delta_{   AB} \delta_{A'B'}$. In particular,
       $\varphi=|x|^2$ satisfies the  assumption (\ref{eq:pseudoconvex})  for $\varphi$ with $c=4$.
   \end{lem}

   To prove this lemma, we introduce complex linear  functions
   \begin{equation}\label{eq:z-k-CF} (z_{AA'}):=\left(
                             \begin{array}{ll}
                               z_{00'  } &  {z}_{ 01' } \\
                               z_{10' } &  {z}_{11' } \\
                               \quad \vdots& \quad \vdots\\
                               z_{ (2l )0' } & {z}_{ (2l )1' } \\
                               z_{ (2l+1)0'  } &  {z}_{ (2l+1)1' } \\
                               \quad \vdots& \quad \vdots\\
                             \end{array}
                           \right):=\left(
                                      \begin{array}{ll}
                                    {x_1} -\textbf{i} {x_2}  & -  {x_3} +\textbf{i} {x_4}  \\
                                        {x_3}+\textbf{i} {x_4}  &\hskip 3mm  {x_1}+\textbf{i} {x_2}  \\
                                        \qquad \vdots&\qquad\vdots\\
                                        {x_{4l+1 }} -\textbf{i} {x_{4l+2}}  & - {x_{4l+3 }}+\textbf{i} {x_{4l+4} }  \\
                                       {x_{4l+3}} +\textbf{i} {x_{4l+4} }  &\hskip 3mm {x_{4l+1 }} +\textbf{i} {x_{4l+2}}  \\ \qquad \vdots&\qquad\vdots\\
                                      \end{array}
                                    \right),
\end{equation} where $A=0,\ldots, 2n-1$, $A'=0',1'$. $z_{AA'}$ is obtained by replacing $\partial_{x_j}$ in $\overline{Z_{AA'}}$ in (\ref{eq:k-CF}) by $x_j$.
By the following lemma, $z_{AA'}$'s can be viewed as independent variables and $Z_{AA'}$'s are derivatives with respect to these variables formally.
\begin{lem}\label{lem:independent} $Z_{AA'}z_{BB'}=2\delta_{AB}\delta_{A'B'}$.
\end{lem}
\begin{proof}  Assume that $A=2l ,A'=0'$. By (\ref{eq:z-k-CF}), we have\begin{equation*}\begin{aligned}Z_{(2l )0' }z_{(2l )0'}= (\partial_{x_{4l+1}}+\textbf{i}\partial_{x_{4l+2}})(x_{4l+1}-\textbf{i}x_{4l+2})=2;\\ Z_{(2l )0'}z_{(2l+1)1' }= (\partial_{x_{4l+1}}+\textbf{i}\partial_{x_{4l+2}})(x_{4l+1}+\textbf{i}x_{4l+2})=0.\end{aligned}\end{equation*}Note that $Z_{(2l )0' }$ is a differential operator with respect to variables $x_{4l+1}$ and $x_{4l+2}$, while $z_{BB' }$ for $BB'\neq {(2l )0' }$ or ${(2l+1)1' }$ is independent of variables $x_{4l+1 }$ and $x_{4l+2}$. So we get
\begin{equation*}
   Z_{(2l )0' }z_{BB'}=0
\end{equation*}
 for such $BB'$. It is similar to check the result directly for other vectors $ Z_{ (2l )1' } $, $    Z_{ (2l+1)0'  }$ and $ Z_{ (2l+1)1' }$.
\end{proof}

{\it Proof of Lemma \ref{lem:assumption}}.
 Note that  $(\partial_{x_j}\pm\mathbf{i}\partial_{x_k}) |x|^2 = 2({x_j}\pm\mathbf{i} {x_k})$. So $\overline{Z_{AC '}}|x|^2 = 2z_{AC '}$ by definitions of $\overline{Z_{AC '}}$'s and  $z_{AC '}$'s in   (\ref{eq:z-k-CF}).   Then    we have
  \begin{equation*}\begin{split}
   Z^{ A ' }_{ B} \overline{{Z^{B ' }_{ A}}} |x|^2 &= \sum_{D',C'}  Z_{     BD'} \overline{Z_{AC '}} |x|^2\cdot \varepsilon^{C'B'}\varepsilon^{D'A'}
   = 2\sum_{D',C'}  Z_{     BD'}  {z_{AC'}}  \cdot  \varepsilon^{C'B'}\varepsilon^{D'A'} \\&= 4\sum_{D',C'}  \delta_{   AB} \delta_{C'D'}  \cdot\varepsilon^{C'B'}\varepsilon^{D'A'}=4\delta_{   AB} \delta_{A'B'} ,
   \end{split}\end{equation*}by Lemma \ref{lem:independent}.  Here by (\ref{eq:varepsilon-anti}), $ \varepsilon^{C'B'}\varepsilon^{C'A'}=1$ only if $A'=B'$ and $C'$ is different from them. Otherwise, it vanishes. So for any $(\xi_{ A'A })\in \mathbb{C}^{2 } \otimes \mathbb{C}^{2 }$,
 \begin{equation*}\begin{split}
  \sum_{A,B,A',B'} Z^{ A ' }_{ B} \overline{{Z^{B ' }_{ A}}} |x|^2 \cdot \xi_{  A' A }   \overline{\xi_{ B' B }} & =4|\xi|^2 .
   \end{split}\end{equation*}

 \subsection{The associated Laplacian operator $\Box_\varphi$ }

By definition,
\begin{equation*}
   {\rm Dom} (\Box_{\varphi }):= \{ f\in L_\varphi^2( \mathbb{R}^{4n},  \mathscr  V_1  ); f\in {\rm Dom} ({\mathscr D}_0^*) \cap  {\rm Dom}  ({\mathscr D}_1), {\mathscr D}_0^* f\in {\rm Dom} ({\mathscr D}_0)  ,   {\mathscr D}_1 f\in {\rm Dom}  ({\mathscr D}_1^*)  \}.
 \end{equation*}
  We introduce   \begin{equation*}
      \mathcal{ E}_\varphi(  f,g ):= \left\langle {\mathscr D}_0^*  f  ,  {\mathscr D}_0^*  g  \right\rangle_{\varphi }   +  \left\langle {\mathscr D}_1  f  , {\mathscr D}_1   g   \right\rangle_{\varphi }
    \end{equation*}
    for any $f,g\in {\rm Dom}( \mathcal{ E}_\varphi) :={\rm Dom}  ({\mathscr D}_1)\cap {\rm Dom} ({\mathscr D}_0^*) $. By definition of adjoint operators, we have
    \begin{equation}\label{eq:E-box}
       \mathcal{ E}_\varphi(  f,g ) =\langle \Box_{\varphi } f,g \rangle_\varphi
    \end{equation}
    for   any $f\in {\rm Dom} (\Box_{\varphi }),g\in {\rm Dom}( \mathcal{ E}_\varphi) $.

Note that for any $F\in {\rm Dom} ({\mathscr D}_0)$, we have ${\mathscr D}_0 F\in {\rm Dom} ({\mathscr D}_1)$ and
 \begin{equation}\label{eq:D1D0}
    {\mathscr D}_1{\mathscr D}_0 F=0  .
 \end{equation}
 This is because ${\mathscr D}_1{\mathscr D}_0 F=0$ for smooth $F$ and the general result follows from the closedness of ${\mathscr D}_0$ and ${\mathscr D}_1$ as differential operators.
     \begin{prop}\label{prop:self-adjoint}
The associated Laplacian operator $\Box_\varphi$ is a densely-defined, closed, self-adjoint and non-negative operator on $L_\varphi^2( \mathbb{R}^{4n},   \mathscr V_1)$.
\end{prop}
    \begin{proof} It is similar to the proof of proposition 4.2.3 of \cite{CS} for $\overline{\partial}$-complex. We give the proof here for completeness.

 As we mentioned before, ${\mathscr D}_0$ and ${\mathscr D}_0^*$ as differential operators are both densely-defined and closed.
 $\Box_\varphi$ is densely-defined  in the same way. For  closedness of $\Box_\varphi$, we need to show that for any $f_n\in {\rm Dom} (\Box_\varphi)$ such that $f_n\longrightarrow f$ in  $L_\varphi^2( \mathbb{R}^{4n},   \mathscr  V_1 )$ and $\Box_\varphi f_n$ converges, we have $f\in {\rm Dom} (\Box_\varphi)$ and $\Box_\varphi f_n \longrightarrow \Box_\varphi f$. Because $f_n\in {\rm Dom} (\Box_\varphi)$, we have
 \begin{equation*}\begin{split}
    \langle \Box_\varphi(f_n-f_m), f_n-f_m\rangle_\varphi&=\langle{\mathscr D}_0{\mathscr D}_0^*(f_n-f_m), f_n-f_m  \rangle_\varphi+\langle{\mathscr D}_1^*{\mathscr D}_1(f_n-f_m),f_n-f_m \rangle_\varphi\\&=\|{\mathscr D}_0^*(f_n-f_m)\|_\varphi^2+\|{\mathscr D}_1(f_n-f_m)\|_\varphi^2,
\end{split} \end{equation*}
 and so ${\mathscr D}_0^*f_n$ and ${\mathscr D}_1f_n$ converge in $L_\varphi^2( \mathbb{R}^{4n},  \mathscr V_0 )$ and $L_\varphi^2( \mathbb{R}^{4n},  \mathscr V_1)$, respectively.  It follows from the closedness of ${\mathscr D}_0^* $ and ${\mathscr D}_1$ that $f\in  {\rm Dom}  ({\mathscr D}_0^*) \cap  {\rm Dom}  ({\mathscr D}_1)$ and
\begin{equation*}
   {\mathscr D}_0^*f_n\longrightarrow {\mathscr D}_0^*f,\qquad  {\mathscr D}_1f_n\longrightarrow {\mathscr D}_1f.
\end{equation*}

 Note that ${\mathscr D}_0{\mathscr D}_0^* f_n$ and $ {\mathscr D}_1^*{\mathscr D}_1 f_n $ are orthogonal to each other by
 \begin{equation*}
    \langle{\mathscr D}_0{\mathscr D}_0^* f_n,{\mathscr D}_1^*{\mathscr D}_1 f_n\rangle_\varphi=\langle{\mathscr D}_1 {\mathscr D}_0{\mathscr D}_0^* f_n,{\mathscr D}_1 f_n\rangle_\varphi=0
 \end{equation*}  by (\ref{eq:D1D0}). So $\Box_\varphi f_n={\mathscr D}_0{\mathscr D}_0^* f_n+ {\mathscr D}_1^*{\mathscr D}_1 f_n$ converges implies that both ${\mathscr D}_0{\mathscr D}_0^* f_n$ and $ {\mathscr D}_1^*{\mathscr D}_1 f_n $ converge. It follows from the closedness of ${\mathscr D}_0 $ and ${\mathscr D}_1^*$ again that
 ${\mathscr D}_0^*f\in   {\rm Dom}  ({\mathscr D}_0)$, ${\mathscr D}_1 f\in ({\mathscr D}_1^*)  $ and
 \begin{equation*}
   {\mathscr D}_0{\mathscr D}_0^*f_n\longrightarrow {\mathscr D}_0{\mathscr D}_0^*f,\qquad {\mathscr D}_1^* {\mathscr D}_1f_n\longrightarrow{\mathscr D}_1^* {\mathscr D}_1f.
\end{equation*} Therefore  $f\in {\rm Dom} (\Box_\varphi)$ and $\Box_\varphi f_n\longrightarrow\Box_\varphi f $. So $\Box_\varphi$ is a closed operator.

  Define
\begin{equation}\label{eq:L1}
  L_1:={\mathscr D}_0{\mathscr D}_0^*  + {\mathscr D}_1^*{\mathscr D}_1+I\qquad {\rm on }\quad {\rm Dom} (\Box_\varphi).
\end{equation}
 It is sufficient to show  that $L_1^{-1}$ is  self-adjoint. By a theorem of Von Neumann (cf. \S 1 in Chapter 8 in \cite{RN}), $(I+{\mathscr D}_0{\mathscr D}_0^* )^{-1}$ and $(1+{\mathscr D}_1^*{\mathscr D}_1)^{-1}$ are automatically both bounded and self-adjoint, and so is
  \begin{equation*}
     Q_1=(I+{\mathscr D}_0{\mathscr D}_0^* )^{-1}+(I+{\mathscr D}_1^*{\mathscr D}_1)^{-1}-I .\end{equation*}
We   claim that $Q_1=L_1^{-1}$. Since
\begin{equation*}
   (1+{\mathscr D}_0{\mathscr D}_0^* )^{-1}-I=(I-(I+{\mathscr D}_0{\mathscr D}_0^*))(I+{\mathscr D}_0{\mathscr D}_0^* )^{-1}=-{\mathscr D}_0{\mathscr D}_0^* (I+{\mathscr D}_0{\mathscr D}_0^* )^{-1},
\end{equation*}we see that $\mathcal{R}(I+{\mathscr D}_0{\mathscr D}_0^* )^{-1}\subset {\rm Dom} ({\mathscr D}_0{\mathscr D}_0^*)$. Similarly, $\mathcal{R}(I+{\mathscr D}_1^*{\mathscr D}_1)^{-1}\subset {\rm Dom} ({\mathscr D}_1^*{\mathscr D}_1)$, and so
\begin{equation}\label{eq:Q1}
  Q_1=(I+{\mathscr D}_1^*{\mathscr D}_1)^{-1}-{\mathscr D}_0{\mathscr D}_0^* (I+{\mathscr D}_0{\mathscr D}_0^* )^{-1}.
\end{equation}
Since ${\mathscr D}_1{\mathscr D}_0=0$ by (\ref{eq:D1D0}), we have $\mathcal{R}(Q_1)\subset {\rm Dom} ({\mathscr D}_1^*{\mathscr D}_1)$ and ${\mathscr D}_1^*{\mathscr D}_1Q_1= {\mathscr D}_1^*{\mathscr D}_1 (I+{\mathscr D}_1^*{\mathscr D}_1)^{-1}$. Similarly $\mathcal{R}(Q_1)\subset {\rm Dom} ({\mathscr D}_0{\mathscr D}_0^*)$ and ${\mathscr D}_0{\mathscr D}_0^*Q_1= {\mathscr D}_0{\mathscr D}_0^*(I+{\mathscr D}_0{\mathscr D}_0^*)^{-1}$. Consequently, $\mathcal{R}(Q_1)\subset {\rm Dom} ( L_1)$ and
\begin{equation*}
   L_1 Q_1={\mathscr D}_1^*{\mathscr D}_1 (I+{\mathscr D}_1^*{\mathscr D}_1)^{-1}+{\mathscr D}_0{\mathscr D}_0^*(I+{\mathscr D}_0{\mathscr D}_0^*)^{-1}+Q_1=I
\end{equation*} by (\ref{eq:Q1}).
This together with the injectivity of $L_1$   implies that $ L_1^{-1}= Q_1$. Thus $L_1^{-1}$ is  self-adjoint. So   is its inverse $L_1$ (cf. \S 2 in Chapter 8 in \cite{RN} for this general property).       \end{proof}

 \subsection{The canonical solution operator}

   {\it Proof of Theorem \ref{thm:canonical}.} (1) The weighted $L^2$-estimate (\ref{eq:L2-n}) implies that
      \begin{equation*}
          4 \left \| g \right\|^2_{  \varphi  }  \leq       \left\| {\mathscr D}_0^*g  \right\|^2_{\varphi } + \left\|{\mathscr D}_1g \right\|^2_{\varphi }=( \Box_{\varphi } g,g )_\varphi\leq\left \| \Box_{\varphi } g \right\|_{  \varphi  }\left  \| g \right\|_{  \varphi  } ,
      \end{equation*} for $  g\in {\rm Dom} (\Box_{\varphi })$,  by (\ref{eq:E-box}), i.e.
      \begin{equation}\label{eq:box-estimate}
 4\left  \| g \right\|_{  \varphi  }\leq \left \| \Box_{\varphi } g\right\|_{  \varphi  }.
      \end{equation}
 Thus $\Box_{\varphi }$ is injective. This together with the self-adjointness of  $\Box_{\varphi }$    by Proposition \ref{prop:self-adjoint} implies the density of the  range (cf. \S 2 in Chapter 8 in \cite{RN} for this general  property). For   fixed $f\in  L_\varphi^2( \mathbb{R}^{4n}, \mathscr  V_1
)$,  the complex anti-linear functional
 \begin{equation*}
    \lambda_f:\Box_{\varphi } g \longrightarrow \langle f, g\rangle_\varphi\end{equation*}
  is then well-defined on a dense subset $\mathcal{R}(\Box_{\varphi })$ of $ L_\varphi^2( \mathbb{R}^{4n}, \mathscr  V_1
)$. It is finite since
 \begin{equation*}
 | \lambda_f(\Box_{\varphi } g )|=  |\langle f, g\rangle_\varphi|\leq \|f\|_\varphi\|g\|_\varphi\leq \frac 14  \|f\|_\varphi\|\Box_{\varphi } g\|_\varphi
 \end{equation*}
 for any $  g\in {\rm Dom} (\Box_{\varphi })$, by (\ref{eq:box-estimate}). So $\lambda_f$ can be uniquely extended a continuous anti-linear functional on $ L_\varphi^2( \mathbb{R}^{4n},\mathscr  V_1 )$. By the Riesz representation theorem, there exists a unique element $h\in  L_\varphi^2( \mathbb{R}^{4n}, \mathscr  V_1 )$ such that $\lambda_f(F)=\langle h, F\rangle_\varphi$ for any $F\in L_\varphi^2( \mathbb{R}^{4n}, \mathscr  V_1 )$,
 and $\|h\|_\varphi= |\lambda_f|\leq \frac 14  \|f\|_\varphi$. In particular, we have
 \begin{equation*}
 \langle h,\Box_{\varphi }g \rangle_\varphi =  \langle f, g\rangle_\varphi
 \end{equation*} for    any $g\in {\rm Dom} (\Box_{\varphi })$. This implies that $h \in {\rm Dom} (\Box_{\varphi }^*)$ and $\Box_{\varphi }^* h=f$, and so $h \in {\rm Dom} (\Box_{\varphi })$ and $\Box_{\varphi } h=f$ by self-adjointness of $\Box_{\varphi }$. We write $h=N_\varphi f$. Then $\left \|N_\varphi f \right\|_{  \varphi  }\leq\frac 14 \left \| f \right\|_{  \varphi  }$.

      (2)  Since $N_\varphi f\in {\rm Dom} (\Box_{\varphi })$, we have ${\mathscr D}_0^*N_\varphi f\in {\rm Dom} ({\mathscr D}_0)$, ${\mathscr D}_1  N_\varphi f\in {\rm Dom} ({\mathscr D}_1^*)$, and
      \begin{equation}\label{eq:N-f}
          {\mathscr D}_0{\mathscr D}_0^*N_\varphi f=f-  {\mathscr D}_1^*{\mathscr D}_1  N_\varphi f
      \end{equation}by $\Box_{\varphi } N_\varphi f=f$.
      Because $f$ and ${\mathscr D}_0 F$ for any $F\in {\rm Dom} ({\mathscr D}_0)$  are both  ${\mathscr D}_1$-closed, the above identity implies $ {\mathscr D}_1^*{\mathscr D}_1  N_\varphi f\in {\rm Dom} ({\mathscr D}_1)$   and  so
      $
{\mathscr D}_1 {\mathscr D}_1^*{\mathscr D}_1  N_\varphi f=0$  by   ${\mathscr D}_1$ acting in both sides. Then
    \begin{equation*}
     0=\langle  {\mathscr D}_1 {\mathscr D}_1^*{\mathscr D}_1  N_\varphi f, {\mathscr D}_1  N_\varphi f\rangle_\varphi=\left\|    {\mathscr D}_1^*{\mathscr D}_1  N_\varphi f\right\|_\varphi^2,
    \end{equation*}
    i.e. $
       {\mathscr D}_1^*  {\mathscr D}_1 N_\varphi f=0.
   $ Hence ${\mathscr D}_0 {\mathscr D}_0^*N_\varphi f =f$ by (\ref{eq:N-f}). Moreover, we have ${\mathscr D}_0^*N_\varphi f\perp A^2_{(k)}(\mathbb{R}^{4n},\varphi)$ since $( F,{\mathscr D}_0^*N_\varphi f )_\varphi=({\mathscr D}_0 F , N_\varphi f )_\varphi=0$ for any $ F\in A^2_{(k)}(\mathbb{R}^{4n},\varphi)$. The  estimate  (\ref{eq:canonical-est}) follows from
   \begin{equation*}
    \| {\mathscr D}_0^*N_\varphi f\|_\varphi^2+ \| {\mathscr D}_1 N_\varphi f\|_\varphi^2  =\langle \Box_{\varphi } N_\varphi f,N_\varphi f \rangle_\varphi\leq \frac 14   \|  f\|_\varphi^2.
   \end{equation*}

   \begin{cor}\label{cor:Bergman} The weighted $k$-Bergman projection formula  (\ref{eq:Bergman-proj}) holds.
\end{cor}
\begin{proof} For $f\in{\rm Dom} ({\mathscr D}_0 ) $, ${\mathscr D}_0 f$ is  automatically ${\mathscr D}_1$-closed.  Apply Theorem \ref{thm:canonical} to ${\mathscr D}_0 f$ to get the canonical solution ${\mathscr D}_0^*N_\varphi{\mathscr D}_0f$   orthogonal to $A^2_{(k)}(\mathbb{R}^{4n},\varphi)$. So $f-{\mathscr D}_0^*N_\varphi{\mathscr D}_0f\in A^2_{(k)}(\mathbb{R}^{4n},\varphi)$ by ${\mathscr D}_0(f-{\mathscr D}_0^*N_\varphi{\mathscr D}_0f)=0$, and is exactly the projection of $f$ to the weighted $k$-Bergman space.
\end{proof}

\begin{rem}
As in \cite{Wa10}, we can use Theorem \ref{thm:canonical} to get compactly supported solution to
the nonhomogeneous     $k$-Cauchy-Fueter equation (\ref{eq:w-kCF})   for ${\mathscr D}_1$-closed  $f\in C_0^1 \left(\mathbb{R}^{4n},  \mathscr V_1  \right)  $, which implies   Hartogs' phenomenon  for
  $k$-regular functions.
  \end{rem}

      \section{Decay of canonical solutions and    the weighted $k$-Bergman   kernel }
       \subsection{The   weighted $k$-Bergman projection and   kernel}
For $f\in L_\varphi^2( \Omega,  \mathscr  V_0  )$,
      it has $k+1$ independent components $f_{0'0'\ldots 0'0'}, f_{1'0'\ldots 0'0'}
,\hdots,f_{1'1'\ldots 1'1'}$. We write
\begin{equation}\label{eq:Dk}
    f=\left(\begin{array}{c}f_{0'0'\ldots 0'0'}\\f_{1'0'\ldots 0'0'}
\\ \vdots \\f_{1'1'\ldots 1'1'}
\end{array}
\right)=\left(\begin{array}{c}f_{0 }\\f_{1 }
\\\vdots\\f_{k }
\end{array}
\right),
\end{equation}
where $f_{j }:=f_{1' \ldots 1' 0'\ldots   0'}$ with $j$ indices  to be  $1' $.

Note that for a sequence of $k$-regular functions $F_n\in L_\varphi^2( \Omega,  \mathscr  V_0 )$ (i.e. ${\mathscr D}_0F_n=0$), if $F_n\longrightarrow F$ in $L_\varphi^2( \Omega,  \mathscr  V_0 )$, we have ${\mathscr D}_0F =0$ by the closedness of ${\mathscr D}_0$. So  $A^2_{(k)}(\mathbb{R}^{4n}, \varphi)$ is a closed   subspace of
$L_\varphi^2( \Omega,  \mathscr  V_0 )$.  If $\{\psi_{\alpha}\}$ is an orthonormal basis of the space   $A^2_{(k)}(\mathbb{R}^{4n}, \varphi)$,
        the weighted $k$-Bergman projection $P$ can be write as
      $
         Pf=\sum_\alpha \langle f,\psi_{\alpha}\rangle_\varphi \psi_{\alpha}.
       $

      \begin{prop} \label{prop:harmonic}
       If $f\in L_\varphi^2( \Omega,  \odot^{k
}\mathbb{C}^{2 } )$ is   $k$-regular, then each component of $f$ is harmonic.
      \end{prop}
      \begin{proof} It follows from
     \begin{equation}\label{eq:D0D0-k}
 \overline{\mathscr   D_0}^{ t} \mathscr   D_0 f = \left(\begin{array}{cccc c} \triangle&
0& \cdots&0
&0 \\0&
2\triangle&
\cdots&0
&0 \\
\vdots  &\vdots &\vdots  & \vdots & \vdots
\\0&0
&
\cdots&2\triangle
&0 \\0&0
&
\cdots&0&\triangle
\end{array}\right)\left(\begin{array}{c}f_{0 }\\f_{1 }
\\\vdots\\f_{k }
\end{array}
\right)
\end{equation}
 where $      \Delta := \partial_{x_{1}}^2+ \partial_{x_{ 2}}^2+
\cdots+ \partial_{x_{ 4n}}^2 $. See lemma 3.3 of \cite{WR} for this identity.
      \end{proof}

By Proposition \ref{prop:harmonic},   each component of  a $k$-regular function  is smooth. So for a fixed point $x\in \mathbb{R}^{4n}$,
we can define  complex  linear functionals
 \begin{equation*}
   l_j(f)=f_j(x)
 \end{equation*} for $f\in A^2_{(k)}(\mathbb{R}^{4n}, \varphi)$, $j=0,\ldots k$. Since $f_j$ is harmonic by Proposition \ref{prop:harmonic}, we see that
 \begin{equation}\label{eq:mean-value}
    |f_j(x)|=\left|\frac 1{|B(x,1)|}\int_{B(x,1)}f_j(y)dV(y)\right|\leq \frac 1{|B(x,1)|}\|f\|_\varphi\left(\int_{B(x,1)}e^{2\varphi(y)}dV(y)\right)^{\frac 12}\leq C_x\|f\|_\varphi,
 \end{equation} where $C_x$ only depends on $x$, not on $f$. Consequently,  linear functionals $l_j$ are bounded on $A^2_{(k)}(\mathbb{R}^{4n}, \varphi)$. By the Riesz representation theorem, there exists $K_j(\cdot,x)\in A^2_{(k)}(\mathbb{R}^{4n}, \varphi)$ such that
 \begin{equation*}
   f_j(x)=\langle f, K_j(\cdot,x)\rangle_\varphi=\sum_{l=0}^{k}\int_{\mathbb{R}^{4n}}f_l(y)\overline{K_{jl}(y,x)}e^{-2\varphi}dV.
 \end{equation*} It is obvious that  $\langle g, K_j(\cdot,x)\rangle_\varphi=0$ for any $g\perp A^2_{(k)}(\mathbb{R}^{4n}, \varphi)$.
So  $K(x,y)=\left(\overline{K_{jl}(y, x)}\right)$ is the kernel of the weighted $k$-Bergman projection $P$, which is
   a $(k+1)\times(k+1)$ matrix  anti-$k$-regular in $y$.
 Then the integral formula (\ref{eq:Bergman-kernel}) holds.  Since an orthogonal projection $P$ is self-adjoint on $ L_\varphi^2( \mathbb{R}^{4n}, \mathscr  V_0
)$, $K$ has the Hermitian
property $K(x,y) = \overline{K(y, x)}^t$, and so
 $ K(x,y)$ is $k$-regular in $x$.

  \subsection{A  localized a priori estimate and   Caccioppoli-type estimate}It is known that
        the    Caccioppoli-type estimate   holds for many systems of PDEs of the divergence form by establishing
  localized  a priori estimate of the following type.
      \begin{prop} \label{prop:localized-L^2}  There exists an absolute constant $C_0>0$ such that for any $f\in {\rm Dom } (\Box_\varphi)$ and real bounded  Lipschitzian function $\eta$, we have estimates
         \begin{equation}\label{eq:localized-L^2}\begin{split}
       \left\|\eta {\mathscr D}_1 f \right\|_{\varphi }^2 + \left\|\eta {\mathscr D}_0^*  f  \right\|_{\varphi }^2 &\leq   C_0\left ( \left\||d\eta|\cdot f\right\|_{\varphi }^2 + |\left\langle \eta^2     f , \Box_\varphi  f\right\rangle_{\varphi }|\right),
   \\
      \mathcal{ E}_\varphi(  \eta f, \eta f ) &\leq   C_0\left ( \left\||d\eta|\cdot f\right\|_{\varphi }^2 +| \left\langle \eta^2     f , \Box_\varphi  f\right\rangle_{\varphi }|\right),
 \end{split}  \end{equation} where $|d\eta|^2=\sum_{j=1}^{4n} |\frac {\partial \eta}{\partial x_j}|^2$.
      \end{prop}
      \begin{proof}
       Note that
        \begin{equation*}\label{eq:delta-eta-f}
         \delta_{A_1' }^{ A}  (\eta f_{  A_2'\ldots A_k'  A  })=\eta\delta_{ A_1' }^{ A}   f_{ A_2'\ldots A_k'  A } +  Z_{ A_1' }^{ A}  \eta \cdot f_{ A_2'\ldots A_k' A }  \end{equation*}
by     $\delta_{ A_1 ' }^{ A}=Z_{ A_1' }^{ A}-2Z_{ A_1' }^{ A}\varphi$ in (\ref{eq:adjoint-delta00}). Then taking summation over $A$ and symmetrising $(A_1'\ldots A_k' )$, we get
       \begin{equation}\label{eq:D0*-eta-f}
  [{\mathscr D}_0^* (\eta f)]_{    A_1'\ldots A_k'  }=\eta [{\mathscr D}_0^*  (f)]_{    A_1'\ldots A_k'  } +\sum_{A=0 }^{2n-1} Z_{ ( A_1 ' }^{ A}\eta \cdot f_{  A_2'\ldots A_k')   A  }.
\end{equation}
  On the other hand, for fixed $A_1'\ldots A_k'$, we have
 \begin{equation}\label{eq:Z-eta-f}\begin{split}
   \left| \sum_{A } Z_{ ( A_1 ' }^{ A}\eta \cdot f_{ A_2'\ldots A_k') A }\right|&=\frac 1k \left| \sum_{s=1}^k \sum_{A }
   Z_{ A_s ' }^{ A}  \eta \cdot f_{  \ldots A_1'\ldots A_k' A}\right|\\&\leq \frac 1k   \sum_{s=1}^k  \left (\sum_{A }\left |Z_{ A_s ' }^{ A}\eta\right |^2\right)^{\frac 12}     \left (\sum_{A } |f_{A_1' \ldots \widehat{A}_s'\ldots A_k' A  }|^2\right)^{\frac 12},
    \end{split} \end{equation} by using (\ref{eq:sym-1}) and Cauchy-Schwarz inequality and $f$ symmetric in the primed indices.
    Note that it directly follows from definition (\ref{eq:k-CF}) of $Z_{ A A  ' }$'s that
    \begin{equation*}
     \sum_{A=0 }^{2n-1}\left |Z_{ A A  ' } \eta\right |^2=  |d\eta|^2
    \end{equation*} for fixed $A'=0'$ or $1'$. Then by raising indices, we get
    \begin{equation*}
       \sum_{A=0 }^{2n-1}\left |Z_{ 0 ' }^A \eta\right |^2= \sum_{A=0 }^{2n-1}\left |\overline{Z^{ 0 ' }_A }\eta\right |^2=  \sum_{A=0 }^{2n-1}\left |  Z_{ A 1 ' }\eta\right |^2=  |d\eta|^2,
    \end{equation*}
       and so is the sum of $  |Z_{ 1 ' }^A \eta  |^2$. Apply these to (\ref{eq:Z-eta-f}) to get
    \begin{equation}\label{eq:sum-Z-eta-f}
    \sum_{  A_1', \ldots,A_k'}  \left\| \sum_{A } Z_{ ( A_1 ' }^{ A}\eta \cdot f_{ A_2'\ldots A_k') A }\right\|^2_\varphi \leq 2 \| |d\eta|\cdot f \|^2_\varphi.
    \end{equation}
  Thus we get the estimate
\begin{equation*}
  \| {\mathscr D}_0^*  (\eta f)\|^2_\varphi\leq \| \eta {\mathscr D}_0  (f)\|^2_\varphi +  2 \| |d\eta|\cdot f \|^2_\varphi,
\end{equation*} by (\ref{eq:D0*-eta-f}),  and simultaneously,
    \begin{equation}\label{eq:caciopoli1}\begin{split}
         \left\|\eta {\mathscr D}_0^* f \right\|_{\varphi }^2& \leq    \left\| {\mathscr D}_0^*(\eta f )\right\|_{\varphi }^2+2\left\||d\eta|\cdot f\right\|_{\varphi }^2.
    \end{split} \end{equation}

  Note that by (\ref{eq:D0*-eta-f}) again,  we get
     \begin{equation*}\begin{split}
       \left\| {\mathscr D}_0^*(\eta f )\right\|_{\varphi }^2 &
        = \sum_{    A_1'\ldots A_k'  } \left( {\mathscr D}_0^*(\eta f )_{    A_1'\ldots A_k'  } ,\sum_{A } Z_{ ( A_1 ' }^{ A}\eta \cdot f_{ A_2'\ldots A_k') A  }  + \eta ( {\mathscr D}_0^*  f)_{    A_1'\ldots A_k'  } \right)_{\varphi }\\&
        =\sum_{    A_1'\ldots A_k'  } \left( {\mathscr D}_0^*(\eta f )_{    A_1'\ldots A_k'  } ,\sum_{A } Z_{ ( A_1 ' }^{ A}\eta \cdot f_{ A_2'\ldots A_k') A }  \right)_{\varphi }+ \left\langle {\mathscr D}_0^*(\eta f ) , \eta {\mathscr D}_0^*  f \right\rangle_{\varphi }
        \\&
        \leq  \kappa \left\| {\mathscr D}_0^*(\eta f )\right\|_{\varphi }^2 +\frac 1{ \kappa}   \left\||d\eta|\cdot f\right\|_{\varphi }^2+  \left\langle \eta f  ,{\mathscr D}_0( \eta {\mathscr D}_0^*  f) \right\rangle_{\varphi }
    \end{split} \end{equation*} by using estimates (\ref{eq:Z-eta-f})-(\ref{eq:sum-Z-eta-f}) and  the trivial inequality $2|ab|\leq \kappa |a|^2+ \frac 1{ \kappa} |b|^2$ for any $\kappa>0$.
    Thus if we choose $\kappa=1/2$, we get
    \begin{equation}\label{eq:caciopoli2}\begin{split}
         \left\| {\mathscr D}_0^*(\eta f )\right\|_{\varphi }^2 &
                \leq    4  \left\||d\eta|\cdot f\right\|_{\varphi }^2+  2\langle  \eta f  , {\mathscr D}_0 (\eta {\mathscr D}_0^*  f)   \rangle_{\varphi }  .
    \end{split} \end{equation}   But
        \begin{equation}\label{eq:caciopoli2'}\begin{split}
           |\langle  \eta f  , {\mathscr D}_0 (\eta {\mathscr D}_0^*  f)   \rangle_{\varphi } | & \leq  | \langle \eta f  , \eta {\mathscr D}_0  {\mathscr D}_0^*  f    \rangle_{\varphi } |+ \sum_{ A ,  A_2',\ldots, A_k'  }\left |\left(  \eta f_{ A_2'\ldots A_k' A } , \sum_{    A_1'} Z^{   A_1 ' }_{ A}\eta \cdot ({\mathscr D}_0^*f)_{   A_1'\ldots A_k'   }    \right)_{\varphi }\right|\\&
           \leq   |\langle \eta^2  f  ,   {\mathscr D}_0  {\mathscr D}_0^*  f    \rangle_{\varphi } |+ \sum_{   A_1',\ldots, A_k'  }\sum_{    A }\left |\left( \overline{Z^{   A_1 ' }_{ A} \eta }f_{ A_2'\ldots A_k' A } ,\eta  ({\mathscr D}_0^*f)_{   A_1'\ldots A_k'   }    \right)_{\varphi }\right|\\&
           \leq   | \langle \eta^2  f  ,   {\mathscr D}_0  {\mathscr D}_0^*  f    \rangle_{\varphi } | +\frac 1{ \kappa} \left\||d\eta|\cdot f\right\|_{\varphi }^2+\kappa\left\|\eta {\mathscr D}_0^*  f \right\|_{\varphi }^2
    \end{split} \end{equation}
   by applying estimates similar to (\ref{eq:Z-eta-f})-(\ref{eq:sum-Z-eta-f}) in the third inequality. Now
    Substitute
   (\ref{eq:caciopoli2'}) to (\ref{eq:caciopoli2}) and using (\ref{eq:caciopoli1}) to control the term $\kappa\left\|\eta {\mathscr D}_0^*  f \right\|_{\varphi }^2$, we find that there exists a constant $C_0>0$ such that
   \begin{equation*} \begin{split}
         \left\|{\mathscr D}_0^*(\eta f ) \right\|_{\varphi }^2
                 \leq   C_0 \left ( \left\||d\eta|\cdot f\right\|_{\varphi }^2 + |\left\langle\eta^2     f , {\mathscr D}_0 {\mathscr D}_0^*  f\right\rangle_{\varphi }|\right).
    \end{split} \end{equation*}
  Similarly,
 \begin{equation*}
   {\mathscr D}_1  (\eta f)_{ A B  A_2'\ldots A_k'  }=\eta ({\mathscr D}_1   f)_{ AB   A_2'\ldots A_k'  } +2\sum_{  A_1'=0',1' } Z^{ A_1 '}_{  [A}\eta \cdot f_{ B]  A_1'\ldots A_k'   }
\end{equation*}by definition,
and so
     \begin{equation*} \begin{split} &\| {\mathscr D}_1  (\eta f)\|^2_\varphi\leq\| \eta {\mathscr D}_1  (f)\|^2_\varphi +  4n \| |d\eta|\cdot f \|^2_\varphi,\\&
         \left\|\eta {\mathscr D}_1  f  \right\|_{\varphi }^2
                 \leq   C_0\left ( \left\||d\eta|\cdot f\right\|_{\varphi }^2 + |\left\langle \eta^2     f , {\mathscr D}_1^*{\mathscr D}_1  f\right\rangle_{\varphi }|\right).
    \end{split} \end{equation*} The result  follows.
      \end{proof}

As a corollary, we get Caccioppoli-type estimate.
   \begin{prop}\label{prop:Caccioppoli} Suppose that $\varphi(x)=|x|^2$. If $\Box_\varphi F=0$ on $B(x,R)\subset \mathbb{R}^{4n}$, then for $r<R$, we have
   \begin{equation*}
     \int_{B(x,r)} |{\mathscr D}_0^*   F |^2e^{-2\varphi }dV
                 \leq    \frac C{(R-r)^2} \int_{B(x,R)}  | F  |^2e^{-2\varphi }dV
   \end{equation*} for some constant $C$ only depending on $n$, $k$, $R$ and $r$.
         \end{prop}
         \begin{proof} Let $\eta$ be a $C_0^\infty(B(x,R))$ function such that $\eta\equiv 1$ on $B(x,r)$.  By the localized  a priori estimate (\ref{eq:localized-L^2})
    in Proposition \ref{prop:localized-L^2}, we get
            \begin{equation*} \begin{split}
         \left\|\chi_{B(x,r)}  {\mathscr D}_0^*   F \right\|_{\varphi }^2&\leq  \left\| \eta {\mathscr D}_0^*   F \right\|_{\varphi }^2
                 \leq   C_0\left ( \left\||d\eta|\cdot  F\right\|_{\varphi }^2 + |\left\langle\eta^2      F , \Box_\varphi F\right\rangle_{\varphi }|\right)=  C_0 \left\| d \eta\right\|_{\infty }^2  \left\|\chi_{B(x,R)}\cdot  F\right\|_{\varphi }^2
    \end{split} \end{equation*}
    since $\Box_\varphi  F=0$ on supp$\,\eta$ and $d\eta$ is supported in $B(x,R)$. The result follows by choosing $\eta$.
         \end{proof}

        \subsection{Decay of canonical solutions and    the weighted $k$-Bergman   kernel}
\begin{thm} \label{thm:decay0}
   Suppose that   $\varphi(x)=|x|^2$,  $k=2,3,\ldots$,  and that  $f\in L_\varphi^2( \mathbb{R}^{4n}, \mathscr  V_1
)$ is compactly supported in $B(y,r_0)$. Then the   canonical solution $ u={\mathscr D}_0^*N_\varphi f$   has the following pointwise estimate: there exists $\varepsilon>0$
   only depending on $r_0$  and constant $C >0$ only depending on $n$, $k$ and $\varepsilon$ such that
   \begin{equation}\label{eq:decay0}
     |u(x)|\leq C e^{ |x|^2+\frac \varepsilon 2|x|-\varepsilon|x-y|}\|f\|_\varphi
   \end{equation}
   for any $x$ such that $|x-y|>r_0+2$.
\end{thm}
    \begin{proof}
         For the canonical solution $u={\mathscr D}_0^*N_\varphi f$, we have ${\mathscr D}_0 u=f$  vanishing  outside of $B(y,r_0)$. Consequently, each component of $u$ is   harmonic   outside of $B(y,r_0)$ by Proposition \ref{prop:harmonic}.
         By the mean value formula for harmonic functions,    we get
         \begin{equation}\label{eq:mean}\begin{split}
           | u(x)|^2 &= \left|\frac 1{|B(x,\delta)|}\int_{B(x,\delta)} u(x') dV \right|^2
           \\&\leq \frac 1{|B(x,\delta)|^2} \int_{B(x,\delta)}  \left|u(x')\right|^2 e^{-2|x'|^2 }dV(x')
           \cdot\int_{B(x,\delta)}  e^{2 |x'|^2 }dV(x')
          \\&\leq C_\delta' e^{2|x|^2+4\delta|x|}\int_{B(x,1)}| N_\varphi f(x')  |^2e^{-2|x'|^2 }dV(x')
         \end{split}   \end{equation} for some   constant $C_\delta'>0$ only depending on $n$, $\delta<1$ and any $x$ such that $|x-y|>r_0+1$. Here in the last inequality       we
  apply    Caccioppoli-type estimate in Proposition \ref{prop:Caccioppoli} to $F= N_\varphi f$ with $\Box_\varphi N_\varphi f =f=0$ outside of $B(y,r_0)$, and
 $  e^{  |x'|^2 }\leq    e^{  |x |^2+ 2\delta|x|+\delta^2 }$ for $x'\in B(x,\delta)$. We choose $\delta=\frac \varepsilon 4$ for $\varepsilon$ determined later.

 For fixed $x$ outside of $B(y,r_0)$, consider the  Lipschitzian function
  \begin{equation*}
    b(x'):=\min \{ |x'-y|, |x-y|\}.
      \end{equation*}   Let $l:[0,\infty)\rightarrow [0,1]$  be the Lipschitzian function vanishing on $[0, r_0]$, equal to $1$ on $[r_0+1,\infty)$, and affine in between. Set $\eta(x')=l(|x'-y|)$. Applying weighted $L^2$  estimate (\ref{eq:L2-n}) and the localized  a priori  estimate
    in Proposition \ref{prop:localized-L^2} to $ N_\varphi f$ with $\eta$ replaced by $\eta e^{\varepsilon b}$, we get
      \begin{equation*}\begin{split}
         \int_{\mathbb{R}^{4n}}|\eta e^{\varepsilon b}N_\varphi f(x')  |^2e^{-2\varphi }dV(x')&\leq  \mathcal{ E}_\varphi( \eta e^{\varepsilon b} N_\varphi f ,  \eta e^{\varepsilon b} N_\varphi f )\\&
        \leq   C_0 \left\||d(\eta e^{\varepsilon b})|\cdot N_\varphi f\right\|_{\varphi }^2 + C_0 \left( \eta^2 e^{2\varepsilon b}     N_\varphi f , \Box_\varphi  N_\varphi f\right)_{\varphi } \\& \leq   C_0 \int_{\mathbb{R}^{4n}}\left(||d\eta| e^{\varepsilon b}N_\varphi f(x')  |^2  +    (4n \varepsilon)^2  |\eta e^{\varepsilon b}N_\varphi f(x')  |^2\right)e^{-2\varphi }dV(x')
       \end{split}  \end{equation*}since the  Lipschitzian  constant of $b$ is $1$ and   $\Box_\varphi N_\varphi f =f=0$ on the support of $\eta$ ($=B(y,r_0)^c$).
       Hence if we choose $\varepsilon$ sufficiently small (e.g. $ C_0  (4n \varepsilon)^2\leq \frac 12 $), we get
         \begin{equation*}\begin{split}
         \int_{\mathbb{R}^{4n}}|\eta e^{\varepsilon b}N_\varphi f(x')  |^2e^{-2\varphi }dV(x')&\leq   2 C_0 \int_{\mathbb{R}^{4n}}||d\eta| e^{\varepsilon b}N_\varphi f(x')  |^2e^{-2\varphi }dV(x')\\&\leq    C '' \int_{B(y,r_0+1)}| N_\varphi f(x')  |^2e^{-2\varphi }dV(x')
       \end{split}  \end{equation*} for some    constant $C''>0$, by $d\eta$ supported in $B(y,r_0+1)$ and $b $  uniformly bounded on  $B(y,r_0+1)$ ($|b(x')|<r_0+1$). But   $b(x')\geq |x-y|-1 $ for  $x'\in B(x,1)$, and so the above estimate implies that
                \begin{equation*}\begin{split}
         \int_{B(x,1)}|N_\varphi f(x')  |^2e^{-2\varphi }dV(x')&\leq     C'' e^{-2\varepsilon (|y-x|-1)} \int_{B(y,r_0+1)}| N_\varphi f (x')  |^2e^{-2\varphi }dV(x').
       \end{split}  \end{equation*} Substituting this into (\ref{eq:mean}), we get the result by the boundedness of $N_\varphi$ on $L_\varphi^2( \mathbb{R}^{4n},  \mathscr V_1 )$  by Theorem \ref{thm:canonical} (1). \end{proof}

 {\it Proof of Theorem \ref{thm:decay}}.  For fixed $y\in  \mathbb{R}^{4n}$, let $\eta_y$ be a smooth radial function supported in the ball $B(y,\delta)$ ($\delta<1$) such that  $\int \eta_y( y')dV(y')=1$. Set
              \begin{equation}\label{eq:fx}f_y(y')=\left(
                             \begin{array}{c}
                              \vdots\\ 0\\ \eta_y(y')e^{ 2|y'|^2}\\ 0\\ \vdots
                             \end{array}
                           \right)\in L_\varphi^2( {\mathbb R}^{4n},  \mathscr V_0 )
\end{equation} for fixed $j $, where only $j$-th entry is nonvanishing.
Note that
       \begin{equation*}
          Pf_y (x )=\int_{\mathbb{R}^{4n}}K(x ,y')f_y (y')e^{-2|y'|^2 }dV(y')=\int_{\mathbb{R}^{4n}}K(x,y')\left(
                             \begin{array}{c}
                              \vdots\\ 0\\ \eta_y(y') \\ 0\\ \vdots
                             \end{array}
                           \right)dV(y')=
          \left(\begin{array}{l}
            K(x ,y)_{0j} \\ \qquad\vdots\\ K(x ,y)_{{k }j}
\end{array}\right)
       \end{equation*} by applying the  mean value formula for  harmonic functions  to
   each component  of $K(x,\cdot)$, since $\eta_y(\cdot)$ is constant on each sphere centered at $y$. Hence the $j$-th column of $(k+1)\times(k+1)$-matrix $K$ is
       \begin{equation*}
           \left(\begin{array}{l}
            K(x ,y)_{0j} \\ \qquad \vdots\\ K(x ,y)_{{k }j}
\end{array}\right)=Pf_y (x )=f_y(x )-({\mathscr D}_0^*N_\varphi {\mathscr D}_0f_y)(x ),
       \end{equation*}by the identity (\ref{eq:Bergman-proj}).
The exponential decay   of the canonical solution in Theorem \ref{thm:decay0} implies that there exists a   constant $C >0$ only depending on $\varepsilon,   n,k$ such that
       \begin{equation*}
       \left  |({\mathscr D}_0^*N_\varphi {\mathscr D}_0f_y)(x)\right|\leq C e^{|x|^2+\frac \varepsilon 2 |x|-\varepsilon|x-y|}\|{\mathscr D}_0f_y\|_\varphi
       \end{equation*}  for any $x$ such that $|x-y|>3$,
 since $   {\mathscr D}_0   {\mathscr D}_0^*N_\varphi {\mathscr D}_0f_y={\mathscr D}_0f_y$ is   supported in $B(y,1)$. Note that
 $ |{\mathscr D}_0f_y(y')|\leq C_3 e^{2|y'|^2}(|y'|+1)\chi_{B(y,\delta)}$ for some    constant $C_3 >0$ depending on $n,\delta$,  by direct differentiation (\ref{eq:fx}). It is direct to check that  $\|{\mathscr D}_0f_y\|_\varphi\leq C_4 e^{|y|^2+ 5\delta|y|}$ for some    constant $C_4 >0$ depending on $n,\delta$. The result follows by choose small $\delta$.

\begin{rem} \label{rem:decay}Our estimate (\ref{eq:decay})  has an extra factor $e^{\frac \varepsilon 2(|x|+|y|)}$ compared to the estimate
   \begin{equation*}
     |K(x,y)|\leq C e^{|x|^2+ |y|^2 -\varepsilon|x-y|},
   \end{equation*}for the Bergmann kernel  in   complex analysis. But when $|y|$ is large compared to $|x|$, e.g. $| y|\geq 4|x|$,
   \begin{equation*}
      |K(x,y)|\leq Ce^{|x|^2+ |y|^2 -\frac \varepsilon 8| y|},
   \end{equation*} which has similar exponential decay with respect to the measure $e^{-|y|^2}dV$ as in the complex case.
\end{rem}

\end{document}